\documentclass{amsart}
\usepackage{amssymb,mathtools,nicefrac,graphicx}
\usepackage[pdftitle={The Second Adjointness Theorem for reductive p-adic groups},
  pdfauthor={Ralf Meyer, Maarten Solleveld},
  pdfsubject={Mathematics},
  pdfkeywords={representation theory, reductive group, Jacquet stability, adjoint functor, Jacquet restriction, parabolic induction}]{hyperref}
\usepackage[lite]{amsrefs}
\usepackage{enumitem}
\usepackage{microtype}

\newcommand*{\MRref}[2]{ \href{http://www.ams.org/mathscinet-getitem?mr=#1}{MR \textbf{#1}}}
\newcommand*{\arxiv}[1]{\href{http://www.arxiv.org/abs/#1}{arXiv: #1}}

\newtheorem{theorem}{Theorem}[section]

\newtheorem{lem}[theorem]{Lemma}
\newtheorem{proposition}[theorem]{Proposition}

\theoremstyle{remark}
\newtheorem{example}[theorem]{Example}
\newtheorem{remark}[theorem]{Remark}
\theoremstyle{definition}
\newtheorem{definition}[theorem]{Definition}

\DeclareMathOperator{\Smooth}{S}
\DeclareMathOperator{\Rough}{R}

\DeclareMathOperator{\Hom}{Hom}
\DeclareMathOperator{\im}{im}
\DeclareMathOperator{\spann}{span}

\newcommand*{\Jaci}{\textup i}
\newcommand*{\Jact}{\textup i^*{}}
\newcommand*{\Jacr}{\textup r}

\newcommand*{\nb}{\nobreakdash}
\newcommand*{\defeq}{\mathrel{\vcentcolon=}}
\newcommand*{\eqdef}{\mathrel{=\vcentcolon}}

\newcommand*{\abs}[1]{\lvert#1\rvert}
\newcommand*{\idem}[1]{\langle#1\rangle}
\newcommand*{\bigidem}[1]{\bigl\langle#1\bigr\rangle}
\newcommand*{\opp}[1]{\overline{#1}}

\newcommand*{\N}{\mathbb N}
\newcommand*{\Z}{\mathbb Z}
\newcommand*{\Q}{\mathbb Q}
\newcommand*{\R}{\mathbb R}
\newcommand*{\C}{\mathbb C}
\newcommand*{\F}{\mathbb F}

\newcommand*{\valu}{\upsilon}

\newcommand*{\Ccinf}[1][\infty]{\textup C_\textup c^{#1}}

\newcommand*{\Cat}{\mathcal C}
\newcommand*{\Mod}[1]{\mathfrak{Mod}_{#1}}

\newcommand*{\Gl}{\textup{Gl}}
\newcommand*{\Sl}{\textup{Sl}}
\newcommand*{\diff}{\textup d}
\newcommand*{\lefts}{\textup l}
\newcommand*{\rights}{\textup r}
\newcommand*{\Op}{T}

\newcommand*{\Galg}{\mathcal G}
\newcommand*{\ST}{S}
\newcommand*{\Un}{U}
\newcommand*{\UC}[1]{U_{#1}}

\newcommand*{\integ}{\mathcal O}
\newcommand*{\unif}{\varpi}

\newcommand{\ring}{R}

\newcommand{\bt}{\mathcal B(G)}
\newcommand*{\Phr}{\Phi^\textup{red}}

\newcommand*{\Hecke}{\mathcal H}
\newcommand*{\epal}{\epsilon_\alpha}

\newcommand*{\twomatrix}[4]{\begin{pmatrix}#1&#2\\#3&#4\end{pmatrix}}
\newcommand*{\stwomatrix}[4]{\left(\begin{smallmatrix}#1&#2\\#3&#4\end{smallmatrix}\right)}

\begin{document}

\title{The Second Adjointness Theorem for reductive \(p\)-adic groups}
\author{Ralf Meyer}
\email{rameyer@uni-math.gwdg.de}
\author{Maarten Solleveld}
\email{maarten@uni-math.gwdg.de}
\address{Mathematisches Institut and Courant Centre ``Higher order structures''\\Georg-August Universit\"at G\"ottingen\\Bunsenstra{\ss}e 3--5\\37073 G\"ottingen\\Germany}

\begin{abstract}
  We prove that the Jacquet restriction functor for a parabolic subgroup of a reductive group over a non-Archimedean local field is right adjoint to the parabolic induction functor for the opposite parabolic subgroup, in the generality of smooth group representations on \(\ring\)\nb-modules for any unital ring~\(\ring\) in which the residue field characteristic is invertible.\\
\textbf{Correction:} it turned out that the paper contains a mistake (in Lemma \ref{lem:rank_reduction}), which the authors
have been unable to fix. This renders the proof of the main theorem incomplete.
\end{abstract}

\thanks{Supported by the German Research Foundation (Deutsche Forschungsgemeinschaft (DFG)) through the Institutional Strategy of the University of G\"ottingen.}
\maketitle

\section{Introduction}
\label{sec:intro}

Let~\(G\) be a reductive group over a non-Archimedan field~\(\F\).  Let~\(p\) be the charactersitic of the residue field of~\(\F\).  Let~\(P\) be a parabolic subgroup of~\(G\) and~\(M\) a Levi subgroup of~\(P\).  Jacquet defined two functors that play an important role in the smooth representation theory of~\(G\).  On the one hand, there is \emph{parabolic induction~\(\Jaci_P^G\)}, which goes from \(M\)\nb-representations via \(P\)\nb-representations to \(G\)\nb--representations; on the other hand, there is \emph{Jacquet restriction~\(\Jacr^P_G\)}, which goes the other way round.

According to the \emph{First Adjointness Theorem}, \(\Jacr^P_G\)~is left adjoint to~\(\Jaci_P^G\), that is,
\[
\Hom_M (\Jacr^P_G(V),W) \cong \Hom_G(V,\Jaci_P^G (W))
\]
for all smooth representations~\(V\) of~\(G\) and~\(W\) of~\(M\).  This is simply a case of Frobenius reciprocity.  Joseph Bernstein's \emph{Second Adjointness Theorem}, which is much more difficult, asserts that~\(\Jaci_P^G\) is left adjoint to the Jacquet restriction functor for the opposite parabolic subgroup~\(\opp{P}\):
\[
\Hom_G(\Jaci^G_P(W),V) \cong \Hom_M (W, \Jacr_G^{\opp{P}} (V)).
\]
Even Bernstein himself admitted to be suprised by this discovery~\cite{Bernstein:Second_adjointness}.  The depth of the Second Adjointness Theorem is witnessed by the highly non-trivial ingredients in Bernstein's so far unpublished proof.  A later proof by Bushnell~\cite{Bushnell:Localization_Hecke} relies on the Bernstein decomposition.

Bernstein and Bushnell proved the Second Adjointness Theorem only for complex representations.  The situation for smooth representations on vector spaces over fields other than~\(\C\) is more complicated.  Vign\'eras \cite{Vigneras:l-modulaires}*{II.3.15} extended the proof to the case where~\(l\) does not divide the pro-order of the group~\(G\).  If \(l\neq p\) but~\(l\) does divide the pro-order of~\(G\), the Bernstein decomposition is known to fail in some cases.  Nevertheless, if~\(G\) is a classical group, Jean-Fran\c{c}ois Dat~\cite{Dat:Finitude} could establish the Second Adjointness Theorem in such characteristics.  We will use a completely different method to prove it in the general setting of smooth representations on \(\ring\)\nb-modules for any unital ring~\(\ring\) in which~\(p\) is invertible.

As in previous arguments, our proof proceeds via the Stabilisation Theorem.  In Section~\ref{sec:Jacquet} we describe Jacquet's functors in the language of bimodules.  We formulate the Second Adjointness Theorem and two versions of the Stabilisation Theorem and clarify in which sense they are all equivalent.  Many results in this section are also explained in more elementary notation and greater detail in the Master's Thesis of David Guiraud~\cite{Guiraud:Master}.

The heart of this article is Section~\ref{sec:stabilisation}, where we prove the Stabilisation Theorem and hence the Second Adjointness Theorem.  This requires some geometric considerations in the building and a fair amount of Bruhat--Tits theory.  Our proof is effective, that is, provides a quantitative version of the Stabilisation Theorem.

The Second Adjointness Theorem has several important consequences.  With our proof, they become unconditional theorems in greater generality.  We mention some of these applications in Section~\ref{sec:consequences}.  This includes the computation of contragredients of Jacquet induced representations and statements about Jacquet induction and restriction of projective representations and finitely presented representations.  A remarkable consequence established in~\cite{Dat:Finitude} is that the Hecke algebra \(\Hecke(G/\!\!/K,\ring)\) of \(K\)\nb-biinvariant, compactly supported, \(\ring\)\nb-valued functions on~\(G\) is Noetherian if~\(\ring\) is Noetherian, for any compact open subgroup~\(K\) of~\(G\).

\section{Second Adjointness and the Stabilisation Theorem}
\label{sec:Jacquet}

Let~\(\F\) be a non-archimedean local field of residual characteristic~\(p\).  That is, \(\F\) is a finite extension of the field of \(p\)\nb-adic numbers~\(\Q_p\), or a field of Laurent series \(\F_{p^d}[\![t,t^{-1}]\), where~\(\F_{p^d}\) is the field with~\(p^d\) elements.  Let~\(\Galg\) be a connected reductive linear algebraic group defined over~\(\F\), and let \(G = \Galg (\F)\) be its group of \(\F\)\nb-rational points.  This is a totally disconnected locally compact group.

We fix a left Haar measure~\(\mu\) on~\(G\) with \(\mu(K) \in p^\Z\) for every open pro-\(p\)-group \(K \subseteq G\).  Let~\(\ring\) be a unital ring in which~\(p\) is invertible.  For instance, \(\ring\)~may be a field of characteristic not equal to~\(p\).

Given a totally disconnected space~\(X\), let \(\Ccinf(X,\ring)\) be the \(\ring\)\nb-module of compactly supported, locally constant functions \(X\to \ring\).  For \(X = G\) this is an algebra under convolution, which we denote by \(\Hecke(G)\) or \(\Hecke(G,\ring)\).  The idempotent in \(\Hecke(G)\) corresponding to a compact open subgroup~\(K\) (with \(\mu(K) \neq 0 \in \ring\)) is denoted~\(\idem{K}\).

A representation of~\(G\) on an \(\ring\)\nb-module~\(V\) is called \emph{smooth} if every \(v\in V\) is fixed by some open subgroup of~\(G\).  The category of non-degenerate \(\Hecke(G)\)\nb-modules is equivalent to the category of smooth representations of~\(G\) on \(\ring\)\nb-modules, where non-degeneracy means \(\Hecke(G)\cdot V= V\).

Let \(P\subseteq G\) be a parabolic subgroup and let \(P=U\cdot M = M\cdot U\) be its decomposition into a unipotent part~\(U\) and a Levi subgroup~\(M\).  The group~\(M\) is also a reductive linear algebraic group.

We are going to describe the Jacquet functors for~\(P\) as tensor product functors with certain bimodules over \(\Hecke(G)\) and~\(\Hecke(M)\).  As a consequence, both functors have right adjoint functors.  The \emph{Second Adjointness Theorem} identifies the right adjoint of the parabolic induction functor for~\(P\) with the Jacquet restriction functor for the opposite parabolic subgroup~\(\opp{P}\).  We show that this statement is equivalent to the \emph{Stabilisation Theorem}.

\subsection{The Jacquet functors as bimodule tensor products}
\label{sec:Jacquet_bimodule}

To describe the Jacquet functors, we shall use the results in~\cite{Meyer:Smooth} about induction and compact induction, restriction, and coinvariant functors.  Let us begin with the functor~\(\Jacr_G^{\opp{P}}\) for the opposite parabolic subgroup~\(\opp{P}\).

Let~\(V\) be a smooth representation of~\(G\), viewed as a smooth \(\Hecke(G)\)-module.  The functor~\(\Jacr_G^{\opp{P}}\) first restricts the action of~\(G\) on~\(V\) to~\(\opp{P}\) and then takes the space of \(\opp{U}\)\nb-coinvariants, that is, the cokernel of the map
\begin{equation}
  \label{eq:def_coinvariant}
  \Hecke(\opp{U})\otimes V\to V,\qquad
  f\otimes v \mapsto f \cdot v - \int_{\opp{U}} f(\bar u) \,\diff\bar u \cdot v.
\end{equation}
Here \(\opp{P} = M\cdot\opp{U}\) is the Levi decomposition of~\(\opp{P}\).  This coinvariant space inherits a canonical smooth representation of~\(M\), which is then twisted by the modular function \(\delta_{\opp{P}}^{-\nicefrac12} = \delta_P^{\nicefrac12}\).  This function is well-defined and invertible because~\(p\) is invertible in~\(\ring\), see \cite{Vigneras:l-modulaires}*{Section I.2}.

Let~\(\C_1\) denote the one-dimensional trivial representation.  By definition, the \(\opp{U}\)\nb-coinvariant space of~\(V\) is \(\C_1\otimes_{\Hecke(\opp{U})} V\).  The restriction functor from~\(G\) to~\(\opp{P}\) may be written as \(V\mapsto \Hecke(G)\otimes_{\Hecke(G)} V\), where we equip~\(\Hecke(G)\) with the natural \(\Hecke(\opp{P}),\Hecke(G)\)-bimodule structure given by left and right convolution.  This represents the restriction functor because \(\Hecke(G)\otimes_{\Hecke(G)} V \cong V\) for all smooth \(\Hecke(G)\)\nb-modules~\(V\).  Putting both functors together yields
\begin{multline*}
  \Jacr_G^{\opp{P}}(V)
  \cong \C_1 \otimes_{\Hecke(\opp{U})} (\Hecke(G)\otimes_{\Hecke(G)} V)
  \\\cong (\C_1 \otimes_{\Hecke(\opp{U})} \Hecke(G))\otimes_{\Hecke(G)} V
  \cong \Ccinf(\opp{U}\backslash G)\otimes_{\Hecke(G)} V,
\end{multline*}
at least as \(\ring\)\nb-modules.  The right \(\Hecke(G)\)-module structure on \(\Ccinf(\opp{U}\backslash G)\) comes from right convolution and corresponds to the representation \((f\cdot h)(\opp{U}g) \defeq f(\opp{U}gh^{-1})\) because~\(G\) is unimodular.  As left \(\Hecke(M)\)-modules
\[
\C_1 \otimes_{\Hecke(\opp{U})} \Hecke(G) \cong
\C_{\delta_{\opp P}} \otimes \Ccinf(\opp{U}\backslash G),
\]
so the left \(\Hecke(M)\)-module structure on \(\Jacr_G^{\opp{P}} V\) comes from the action of~\(M\) on \(\Ccinf(\opp{U}\backslash G)\) by
\[
(m\cdot f) (\opp{U}g) =  \delta_{\opp{P}}(m)^{\nicefrac12} \cdot f(\opp{U}mg).
\]

Next we turn to the parabolic induction functor~\(\Jaci_P^G\) associated to the parabolic subgroup~\(P\), which maps smooth representations of~\(M\) to smooth representations of~\(G\).  Let~\(W\) be a smooth representation of~\(M\), viewed as a module over \(\Hecke(M)\).  The functor~\(\Jaci_P^G\) first extends the representation from~\(M\) to~\(P\) by letting~\(U\) act trivially; then it twists the action of~\(M\) by the modular function~\(\delta_P^{\nicefrac12}\); finally, it induces from~\(P\) to~\(G\).

Since~\(P\) is cocompact in~\(G\), induction is the same as compact induction.  The compact induction functor is described in \cite{Meyer:Smooth}*{Theorem 4.10} as the tensor product functor \(V\mapsto \Hecke(G)\otimes_{\Hecke(P)} V\), up to a modular function~\(\delta_P\).  Extending the action from~\(M\) to~\(P\) may be written as \(W\mapsto \Hecke(M)\otimes_{\Hecke(M)} W\), where we view \(\Hecke(M)\) as an \(\Hecke(P),\Hecke(M)\)-bimodule using the left action of~\(P\) by \((mu)\cdot f(m') \defeq f(m^{-1}m')\) for \(m,m'\in M\), \(u\in U\).  Ignoring modular functions for the moment, we compute
\begin{multline}
  \label{eq:IPGW}
  \Jaci_P^G(W) = \Hecke(G)\otimes_{\Hecke(P)} (\Hecke(M)\otimes_{\Hecke(M)} W)\\
  \cong \bigl(\Hecke(G)\otimes_{\Hecke(P)} \Hecke(M)\bigr)\otimes_{\Hecke(M)} W
  \cong \Ccinf(G/ U)\otimes_{\Hecke(M)} W.
\end{multline}
We used that \(\Hecke(M) = \Hecke(P) \otimes_{\Hecke(U)} \C_1\) is the \(U\)\nb-coinvariant space of \(\Hecke(P)\), so that
\begin{multline}
  \label{eq:Ucoinvariants}
  \Hecke(G)\otimes_{\Hecke(P)} \Hecke(M)
  \cong \Hecke(G)\otimes_{\Hecke(P)} \Hecke(P)\otimes_{\Hecke(U)}\C_1
  \\\cong \Hecke(G)\otimes_{\Hecke(U)}\C_1
  \cong \Ccinf(G/U).
\end{multline}
The left \(\Hecke(G)\)-module structure on \(\Ccinf(G/U)\) is given simply by the left regular representation, \((g\cdot f) (hU) \defeq f(g^{-1}hU)\).  The right \(\Hecke(M)\)-module structure on \(\Ccinf (G/U)\) in~\eqref{eq:Ucoinvariants} comes from the right translation action of~\(M\), twisted by \(\delta_P^{-1}\).  Taking into account that~\(W\) has been twisted by \(\delta_P^{\nicefrac12}\), we find that the right action of~\(M\) on \(\Ccinf (G/U)\) in~\eqref{eq:IPGW} is \((f\cdot m) (hU) = f(hm^{-1}U)\cdot \delta_P(m)^{-\nicefrac12}\), while~\(G\) acts simply by left translation.

In order to summarise the above statements, we let \(\Mod{A}\) denote the category of smooth \(A\)\nb-modules, that is, left \(A\)\nb-modules~\(X\) with \(A\otimes_A X\cong X\).

\begin{proposition}
  \label{pro:Jacquet_functors}
  The functor
  \[
  \Jacr_G^{\opp{P}}\colon \Mod{\Hecke(G)}\to\Mod{\Hecke(M)}
  \]
  is the tensor product functor for the smooth \(\Hecke(M),\Hecke(G)\)-bimodule \(\Ccinf(\opp{U}\backslash G)\) with module structures from the left and right representations of \(M\) and~\(G\) by
  \[
  (m\cdot f) (h\opp{U}) \defeq f(m^{-1}h\opp{U})\delta_{\opp{P}}(m)^{\nicefrac12},\qquad
  (f\cdot g) (h\opp{U}) \defeq f(hg^{-1}\opp{U}).
  \]

  The functor
  \[
  \Jaci_P^G\colon \Mod{\Hecke(M)}\to\Mod{\Hecke(G)}
  \]
  is the tensor product functor for the smooth \(\Hecke(G),\Hecke(M)\)-bimodule \(\Ccinf(G/U)\) with module structures from the left and right representations of \(G\) and~\(M\) by
  \[
  (g\cdot f) (hU) \defeq f(g^{-1}hU),\qquad
  (f\cdot m) (hU) \defeq f(hm^{-1}U)\delta_P(m)^{-\nicefrac12}.
  \]
\end{proposition}

Recall the \emph{adjoint associativity} isomorphism
\[
\Hom_A(Y\otimes_B V,W) \cong
\Hom_B\bigl(V,\Hom_A(Y,W)\bigr),
\]
where~\(Y\) is an \(A,B\)-bimodule, \(V\) is a \(B\)\nb-module, and~\(W\) is an \(A\)\nb-module.  We may combine this with the smoothening functor~\(\Smooth_B\) for \(B\)\nb-modules to get a functor \(W\mapsto \Smooth_B \bigl(\Hom_A(Y,W)\bigr)\) between categories of non-degenerate modules, where \(\Smooth_B(X) \defeq B\otimes_B X\) and~\(B\) is a self-induced algebra, that is, \(B\otimes_B B\cong B\) (see also~\cite{Meyer:Smooth_rough}).

Since Proposition~\ref{pro:Jacquet_functors} expresses \(\Jaci_P^G\) and~\(\Jacr_G^{\opp{P}}\) as bimodule tensor products, both have a right adjoint.  Recall also that the right adjoint is unique up to natural isomorphism if it exists.

The well-known \emph{First Adjointness Theorem} asserts that the~\(\Jaci_P^G\) is right adjoint to~\(\Jacr_G^P\):
\begin{equation}
  \label{eq:first_adjointness}
  \Hom_{\Hecke(G)}(V, \Jaci_P^G (W)) \cong \Hom_{\Hecke(M)}(\Jacr_G^P (V), W)
\end{equation}
for all representations \(V\) and~\(W\) of \(G\) and~\(M\), respectively.  Adjoint associativity produces another formula for this right adjoint, namely, the functor
\[
W\mapsto \Smooth_{\Hecke(G)} \bigl(\Hom_{\Hecke(M)}(\Ccinf(U\backslash G),W)\bigr)
\]
for smooth \(\Hecke(M)\)-modules~\(W\).  It is an instructive exercise to verify that this functor is naturally isomorphic to~\(\Jaci_P^G\).  The reason is that~\(P\) is cocompact in~\(G\).

\subsection{Formulation of the Second Adjointness Theorem}
\label{sec:second_adjointness_formulate}

Here we are interested in the much deeper statement that~\(\Jacr_G^{\opp{P}}\) is right adjoint to~\(\Jaci_P^G\).  Adjoint associativity shows that~\(\Jaci_P^G\) has the right adjoint functor
\begin{equation}
  \label{eq:Jacquet_ind_adjoint}
  \Jact_G^P\colon V\mapsto \Smooth_{\Hecke(M)}
  \bigl(\Hom_{\Hecke(G)}(\Ccinf(G/U),V)\bigr)
\end{equation}
for a smooth \(\Hecke(G)\)-module~\(V\), where the bimodule structure on \(\Ccinf (G/U)\) is as in Proposition~\ref{pro:Jacquet_functors}.  The Second Adjointness Theorem is therefore equivalent to a natural isomorphism \(\Jact_G^P \cong \Jacr_G^{\opp{P}}\).

Let us first describe~\(\Jact_G^P\) more concretely.  Since \(\Ccinf(G/U)\) is a quotient of \(\Hecke(G)\) by the integration map
\[
(\pi f)(gU) \defeq \int_U f(gu)\,\diff u,
\]
an \(\Hecke(G)\)-module map \(\Ccinf(G/U)\to V\) yields an \(\Hecke(G)\)-module map \(\Hecke(G)\to V\) as well.  Thus \(\Hom_{\Hecke(G)}(\Ccinf(G/U),V)\) is contained in the \emph{roughening}
\[
\Hom_{\Hecke(G)}(\Hecke(G),V) \eqdef \Rough_{\Hecke(G)} V = \Rough (V).
\]
Recall that the roughening of an \(\Hecke(G)\)-module~\(V\) is the projective limit of the invariant subspaces~\((V^K)\), where~\(K\) runs through the directed set of compact open subgroups of~\(G\) and the map \(V^K\to V^L\) for \(K\subseteq L\) is induced by \(\idem{K}\in\Hecke(G)\).  Here \(\idem{K}\in\Hecke(G)\) for a compact open pro-\(p\)-subgroup \(K\subseteq G\) denotes the projection associated to the normalised Haar measure on~\(K\).  It acts on a representation of~\(G\) by projecting to the space of \(K\)\nb-invariants and annihilating all other irreducible \(K\)\nb-subrepresentations.  We get \(\Hom_{\Hecke(G)}(\Hecke(G),V) \cong \varprojlim V^K\) because \(\Hecke(G) = \varinjlim \Ccinf(G/K)\) and \(\Hom_{\Hecke(G)}(\Ccinf(G/K),V) \cong V^K\).

A map \(\Hecke(G)\to V\) factors through the quotient map \(\pi\colon \Hecke(G)\to\Ccinf(G/U)\) if and only if the corresponding element of \(\Rough(V)\) is \(U\)\nb-invariant.  As a result, \(\Jact_G^P(V)\) is the space of \(M\)\nb-smooth, \(U\)\nb-invariant vectors in the roughening of~\(V\); the action of~\(M\) is the extension of the action of~\(V\) to~\(\Rough(V)\), twisted by \(\delta_P^{-\nicefrac12}\).

In order to compare \(\Jacr_G^{\opp{P}}\) and~\(\Jact_G^P\), we need some more notation.

\begin{definition}
  \label{def:well-placed}
  A compact open subgroup \(K \subseteq G\) is called \emph{well-placed} or \emph{in good position} with respect to \(\{P,\opp{P}\}\) if
  \begin{itemize}
  \item the multiplication map \((K \cap U) \times (K \cap M) \times (K \cap \opp{U}) \to K\) is bijective, and the same holds for any other ordering of the three factors;
  \item \(K\) is a pro-\(p\)-group;
  \end{itemize}
\end{definition}

Thus every subgroup \(H \subseteq K\) is also a pro-\(p\)-group and hence has a Haar measure with values in \(\Z[\nicefrac1p]\) or~\(\ring\).  We may regard this Haar measure as a multiplier~\(\idem{H}\) of \(\Hecke(G)\), which is idempotent and satisfies \(\idem{H} \idem{K} = \idem{K} = \idem{K} \idem{H}\).

Since all open subgroups of \(U\) and~\(\opp{U}\) are pro-\(p\)-groups, \(K\) is a pro-\(p\)-group if and only if \(K\cap M\) is a pro-\(p\)-group.  Any sufficiently small compact open subgroup of \(G\) or~\(M\) is a pro-\(p\)-group.

Bruhat--Tits theory produces examples of sequences~\((K_e)_{e\in\N}\) of compact open subgroups of~\(G\) such that
\begin{itemize}
\item each~\(K_e\) is a normal in \(K \defeq K_0\);
\item the sequence \((K_e)_{e\in\N}\) decreases and is a neighborhood basis of~\(1\) in~\(G\);
\item each~\(K_e\) is in good position with respect to \(\{P,\opp{P}\}\).
\end{itemize}
We abbreviate
\[
K^+_e \defeq K_e \cap U,\qquad
K^0_e \defeq K_e \cap M,\qquad
K^-_e \defeq K_e \cap \opp{U},
\]
so that
\begin{multline}
  \label{eq:Kn}
  K_e = K_e^- K_e^0 K_e^+
  = K_e^- K_e^+ K_e^0
  = K_e^+ K_e^- K_e^0
  \\= K_e^+ K_e^0 K_e^-
  = K_e^0 K_e^+ K_e^-
  = K_e^0 K_e^- K_e^+.
\end{multline}
Notice that \(K \cap M\) normalises~\(K_e^\pm\) for all~\(e\) because~\(M\) normalises \(U\) and~\(\opp{U}\).

Now we define a natural map \(\Jact_G^P(V) \to \Jacr_G^{\opp{P}}(V)\).  Recall that~\(\Jact_G^P(V)\) consists of \(M\)\nb-smooth, \(U\)\nb-invariant elements of~\(\Rough(V)\).  Being \(M\)\nb-smooth means being invariant under~\(K_e^0\) for sufficiently large \(e\in\N\).

\begin{lem}
  \label{lem:invariance_average}
  If \(v\in \Rough(V)\) is invariant under \(U\) and~\(K_e^0\), then \(\idem{K_e^-}v = \idem{K_e}v\) belongs to \(V^{K_e}\subseteq V\subseteq \Rough(V)\).  The class of \(\idem{K_e^-}v\) in the \(\opp{U}\)\nb-coinvariant space of~\(V\) does not depend on the choice of~\(e\).
\end{lem}

\begin{proof}
  Since \(K_e^+\subseteq U\), we have \(v = \idem{K_e^+}v = \idem{K_e^0} \idem{K_e^+}v\).  Hence
  \[
  \idem{K_e^-}v = \idem{K_e^-} \idem{K_e^0} \idem{K_e^+} v = \idem{K_e}v
  \]
  by~\eqref{eq:Kn}.  Since~\(K_e\) is open, \(V^{K_e} = \Rough(V)^{K_e}\), so that \(\idem{K_e^-}v \in V\).  The class of \(\idem{K_e^-}v\) in the coinvariant space~\(V/\opp{U}\) does not depend on~\(e\) because \(K_e^-\subseteq\opp{U}\) for all~\(e\).
\end{proof}

It is straightforward to check that the natural map
\[
\Jact_G^P(V)\to \Jacr_G^{\opp{P}}(V),\qquad
v\mapsto \idem{K_e^-}v,
\]
defined by Lemma~\ref{lem:invariance_average} is \(M\)\nb-equivariant.  The modular factors on both sides agree because \(\delta_P^{-\nicefrac12} = \delta_{\opp{P}}^{\nicefrac12}\).  The following is our main result:

\begin{theorem}[Second Adjointness Theorem]
  \label{the:second_adjointness}
  Let~\(G\) be a reductive group over a non-Archimedean local field with residue field charactersitic~\(p\).  Let~\(P\) be a parabolic subgroup and let~\(\opp{P}\) be its opposite parabolic.  Let~\(\ring\) be a unital ring in which~\(p\) is invertible and let~\(V\) be a smooth representation of~\(G\) on an \(\ring\)\nb-module.  Then the natural map \(\Jact_G^P(V)\to \Jacr_G^{\opp{P}}(V)\) is invertible.
\end{theorem}

We will prove this theorem in Section~\ref{sec:stabilisation}.

\subsection{The Stabilisation Theorem}
\label{sec:stabilisation_formulation}

Let~\(\lambda\) be an element of the centre~\(Z(M)\) of~\(M\) that is strictly positive with respect to \((\opp{P},M)\), which means that
\begin{equation}
  \label{eq:lambdaK+}
  \bigcup_{n\in\N} \lambda^n (K \cap U) \lambda^{-n} = U,\qquad
  \bigcap_{n\in\N} \lambda^{-n} (K \cap U) \lambda^n = \{1\}.
\end{equation}
Then~\(\lambda^{-1}\) is strictly positive with respect to \((P,M)\), that is,
\begin{equation}
  \label{eq:lambdaK-}
  \bigcup_{n\in\N} \lambda^{-n} (K \cap \opp{U}) \lambda^n = \opp{U},\qquad
  \bigcap_{n\in\N} \lambda^n (K \cap \opp{U}) \lambda^{-n} = \{1\}.
\end{equation}
For example, if \(G=\Gl_n\), \(P\)~is the parabolic subgroup of all upper triangular matrices, and~\(M\) is the subgroup of diagonal matrices, then~\(\lambda\) is a diagonal matrix whose entries have strictly decreasing norms.

Given \(g\in G\), we abbreviate
\[
\idem{K g K} \defeq \idem{K} g \idem{K}.
\]
Up to a volume factor, which is invertible in \(\Z[\nicefrac1p]\), \(\idem{K g K}\) is the characteristic function of the double coset~\(KgK\).

\begin{theorem}[Stabilisation Theorem]
  \label{the:stabilisation_general}
  Let~\(V\) be a smooth representation of~\(G\) on an \(\ring\)\nb-module.  Then the sequence of subspaces
  \begin{equation}
    \label{eq:Jacquet_stability}
    \ker \bigl(\idem{K \lambda K}^n\colon V\to V\bigr),\qquad
    \im \bigl(\idem{K \lambda K}^n\colon V\to V\bigr)
  \end{equation}
  stabilises for sufficiently large \(n\in\N\), that is, these subspaces become independent of~\(n\) for sufficiently large~\(n\).
\end{theorem}

William Casselman \cite{Casselman:Representations_padic_groups}*{Section 4} proved this statement for admissible complex representations, Marie-France Vign\'eras \cite{Vigneras:l-modulaires}*{II.3.6} for admissible representations on vector spaces over fields of characteristic different from~\(p\), and Joseph Bernstein~\cite{Bernstein:Second_adjointness} for all smooth complex representations~-- which is much more difficult than the admissible case.  We will establish the Stabilisation Theorem~\ref{the:stabilisation_general} in Section~\ref{sec:stabilisation}.

An \(\ring\)\nb-module homomorphism \(\Op\colon W\to W\) is called \emph{stable} if \(\ker \Op = \ker \Op^2\) and \(\im \Op=\im \Op^2\).  The sequence \(\ker \Op^n\) is increasing and the sequence \(\im \Op^n\) is decreasing, and~\(\Op^2\) is stable if~\(\Op\) is stable.  Hence the Stabilisation Theorem~\ref{the:stabilisation_general} holds if and only if there is \(n\in\N\) for which \(\idem{K\lambda K}^n\)~is stable.

\begin{lem}
  \label{lem:stable_decomposition}
  An \(\ring\)\nb-module homomorphism \(\Op\colon W\to W\) is stable if and only if~\(\Op\) restricts to an invertible map \(\im \Op\to \im \Op\), if and only if \(W= \ker \Op \oplus \im \Op\).
\end{lem}

\begin{proof}
  The stability property \(\im \Op^2=\im \Op\) is equivalent to the surjectivity of the map \(\im \Op\to\im \Op\) induced by~\(\Op\).  We have \(\ker \Op\cap \im \Op = \Op(\ker \Op^2)\).  Thus the stability property \(\ker \Op^2=\ker \Op\) is equivalent to \(\ker \Op\cap \im \Op=\{0\}\), that is, the injectivity of \(\Op|_{\im \Op}\).  Thus~\(\Op\) is stable if and only if \(\Op|_{\im \Op} \colon \im \Op\to \im \Op\) is bijective.  Clearly, this follows if \(W=\ker \Op\oplus \im \Op\).

  Conversely, suppose that~\(\Op\) is stable.  Then \(\ker \Op^2=\ker \Op\) implies \(\ker \Op\cap \im \Op=\{0\}\).  If \(x\in W\), then \(\Op x\in \im \Op = \im \Op^2\), so that \(\Op x=\Op^2y\) for some \(y\in W\), that is, \(x-\Op y\in \ker \Op\).  Thus \(x\in \ker \Op +\im \Op\).
\end{proof}

We may reformulate the Stabilisation Theorem using \(\idem{K\lambda K}^n = \idem{K\lambda^n K}\):

\begin{lem}
  \label{lem:1}
  Let \(K \subseteq G\) be a compact open pro-\(p\)-group which is in good position with respect to \(\{P,\opp{P}\}\).  The following equalities hold in \(\Hecke(G)\):
  \[
  \idem{K} \idem{\lambda K \lambda^{-1}} \dotsm \idem{\lambda^m K \lambda^{-m}}
  = \idem{K \lambda K}^m \lambda^{-m}
  = \idem{K \lambda^m K} \lambda^{-m}
  = \idem{K} \idem{\lambda^m K \lambda^{-m}}.
  \]
\end{lem}

\begin{proof}
  The first and third equalities are trivial.  We prove the second one.  If \(\mu \in Z(M)\) is another element which is strictly positive with respect to \((\opp{P},M)\), then \eqref{eq:lambdaK+}, \eqref{eq:lambdaK-} and~\eqref{eq:Kn} yield
  \begin{multline*}
    \idem{K \lambda K} \idem{K \mu K}
    = \idem{K} \lambda \idem{K} \mu \idem{K}
    = \idem{K} \lambda \idem{K \cap \opp{U}} \idem{K \cap M} \idem{K \cap U} \mu \idem{K}\\
    = \idem{K} \idem{\lambda (K \cap \opp{U}) \lambda^{-1}}
    \idem{\lambda (K \cap M) \lambda^{-1}} \lambda \mu \idem{\mu^{-1} (K \cap U) \mu} \idem{K}
    = \idem{K} \lambda \mu \idem{K}.
  \end{multline*}
  Applying this with \(\mu = \lambda^d\), induction on \(d\in\N\) yields \(\idem{K \lambda K}^m = \idem{K \lambda^m K}\).
\end{proof}

\subsection{Equivalence of the Second Adjointness Theorem and the Stabilisation Theorem}
\label{sec:adjointness_versus_stabilisation}

Our next goal is to establish that the Second Adjointness Theorem is equivalent to the Stabilisation Theorem.  This motivates us to prove the Stabilisation Theorem in Section~\ref{sec:stabilisation}.  More precisely, the logic is a bit more complicated.

The Second Adjointness Theorem follows if the Stabilisation Theorem holds for some cofinal sequence of subgroups in good position~-- this is what we are going to do in Section~\ref{sec:stabilisation}.  Conversely, the Second Adjointness Theorem implies the Stabilisation Theorem for \emph{all} subgroups in good position.

\begin{proposition}
  \label{pro:Jacquet_as_limit}
  Let \(K\subseteq G\) be a compact open subgroup that is in good position with respect to \(\{P,\opp{P}\}\) and let~\(\lambda\) be strictly positive with respect to \((\opp{P},M)\).  Let~\(V\) be a smooth representation of~\(G\).  There are natural isomorphisms
  \begin{align*}
    \Jacr_G^{\opp{P}}(V)^{K \cap M} &\cong \varinjlim \bigl(
    V^K \xrightarrow{\idem{K \lambda K}}
    V^K \xrightarrow{\idem{K \lambda K}}
    V^K \to \dotsb\bigr),\\
    \Jact_G^P(V)^{K \cap M} &\cong \varprojlim \bigl(\dotsb\to
    V^K \xrightarrow{\idem{K \lambda K}}
    V^K \xrightarrow{\idem{K \lambda K}}
    V^K\bigr),
  \end{align*}
  such that the natural map \(\Jact_G^P(V)\to \Jacr_G^{\opp{P}}(V)\) described in Lemma~\textup{\ref{lem:invariance_average}} restricts to the natural map from the projective to the inductive limit of the diagram
  \begin{equation}
    \label{eq:proj_injlim_diagram}
    \dotsb \to
    V^K \xrightarrow{\idem{K \lambda K}}
    V^K \xrightarrow{\idem{K \lambda K}}
    V^K \xrightarrow{\idem{K \lambda K}}
    V^K \xrightarrow{\idem{K \lambda K}}
    V^K \to \dotsb.
  \end{equation}
\end{proposition}

We write \(\varprojlim \Op\) and~\(\varinjlim \Op\) for the projective and inductive limits of the diagram
\begin{equation}
  \label{eq:constant_diagram}
  \dotsb \to W \xrightarrow{\Op} W \xrightarrow{\Op} W
  \xrightarrow{\Op} W \xrightarrow{\Op} W \to \dotsb.
\end{equation}

\begin{proof}
  First we consider the functor~\(\Jacr_G^{\opp{P}}\).  For any subgroup \(H \subseteq G\) we put
  \[
  V(H) \defeq \spann {}\{ v - h \cdot v \mid v\in V,\ h\in H \}.
  \]
  Equation~\eqref{eq:lambdaK-} provides \(n(e)\in\N\) such that the sequence of groups
  \[
  H_e^- \defeq \lambda^{-n(e)} K_e^- \lambda^{n(e)},\qquad e\in\N
  \]
  increases and has union~\(\opp{U}\).  Thus the \(\opp{U}\)\nb-coinvariant space is the inductive limit
  \[
  V \bigm/ V(\opp{U}) = \lim_{e\to\infty} V \bigm/ V(H_e^-)
  \]
  of the coinvariant spaces for~\(H_e^-\).  We may also assume \(n(0)=0\).

  Since the groups~\(H_e^-\) are compact, \(V \bigm/ V (H_e^-)\) is naturally isomorphic to \(V^{H_e^-}\).  Under this isomorphism, the canonical maps \(V \bigm/ V (H_{e-1}^-) \to V \bigm/ V(H_e^-)\) and \(V\to V \bigm/ V(H_e^-)\) correspond to integration over~\(H_e^-\).

  If~\(V\) is a smooth representation of~\(G\), then it is smooth as a representation of~\(P\).  Hence it is the union (and in particular inductive limit) of the invariant spaces~\(V^{K^0_e H_e^+}\), where \(H_e^+ \defeq \lambda^{-n(e)} K_e^+ \lambda^{n(e)}\).  Notice that
  \[
  K^0_e H_e^+ H_e^- = \lambda^{-n(e)} K^0_e K^+_e K_e^- \lambda^{n(e)} =
  \lambda^{-n(e)} K_e \lambda^{n(e)}
  \]
  is a subgroup.  Hence integration over~\(H_e^-\) maps the space of \(K^0_e H_e^+\)\nb-invariants into itself.  As a consequence, \(\Jacr_G^{\opp{P}}(V)\) is isomorphic to the inductive limit of the subspaces of \(K^0_e H_e^+ H_e^-\)-invariants for \(e\to\infty\), where the structure map is integration over~\(H_e^-\).  Since \(K^0_e H^+_e \subseteq K^0_{e-1} H^+_{e-1}\), this structure map is equal to integration over the whole group \(H^-_e K^0_e H^+_e\).  Thus
  \[
  \Jacr_G^{\opp{P}} (V) \cong \varinjlim V^{\lambda^{-n(e)} K_e \lambda^{n(e)}},
  \]
  where the maps are given by \(\lambda^{-n(e)}\idem{K_e}\lambda^{n(e)}\) on \(V^{\lambda^{n(e-1)}K_{e-1}\lambda^{n(e-1)}}\).  Similar computations yield
  \[
  \Jacr_G^{\opp{P}}(V)^{K \cap M}
  \cong \varinjlim_e V^{\lambda^{-n(e)} K \lambda^{n(e)}}
  \cong \varinjlim_n V^{\lambda^{-n}K \lambda^n}.
  \]
  Replacing~\(n(e)\) by~\(n\) does not change the colimit because the sequence~\((n(e))_{e\in\N}\) is cofinal in~\(\N\).  Since the groups~\(\lambda^{-n}K \lambda^n\) are all conjugate, their fixed-point subspaces are isomorphic via \(\lambda^{\pm n}\).  This yields the desired description of \(\Jacr_G^{\opp{P}}(V)^{K \cap M}\) as \(\varinjlim {}\idem{K \lambda K}\).

  Now we consider the functor~\(\Jact_G^P\).  Recall that \(\Jact_G^P(V)\) is the subspace of \(M\)\nb-smooth, \(U\)\nb-invariant elements in the roughening \(\Rough (V)\) of~\(V\).  We restrict attention to the subspace of \((K \cap M)\)-invariants.  Then \(M\)\nb-smoothness becomes automatic, so that \(\Jact_G^P(V)^{K \cap M}\) is the subspace of \((K \cap M) U\)\nb-invariants in \(\Rough(V)\).  As
  \[
  \Rough(V) \cong \varprojlim ( \dotsb\to V^{K_{e+1}} \xrightarrow{\idem{K_e}}
  V^{K_e} \to \dotsb \to V^{K_1} \xrightarrow{\idem{K_0}} V^{K_0} ),
  \]
  an element of \(\Rough(V)\) is a sequence of \(x_e \in V^{K_e}\) with \(\idem{K_e}\cdot x_{e+1} = x_e\).

  Equation~\eqref{eq:lambdaK-} yields an increasing sequence \(m(e) \in \N\) with \(m(0)=0\) and \(\lambda^{m(e)} (K \cap \opp{U}) \lambda^{-m(e)} \subseteq K_e^-\).  Then
  \begin{multline*}
    K_e (K \cap M) U
    = K_e^- (K \cap M) U
    \supseteq \lambda^{m(e)} (K \cap \opp{U}) \lambda^{-m(e)}
    (K \cap M) U
    \\\supseteq \lambda^{m(e)} (K \cap \opp{U}) (K \cap M)
    (K \cap U) \lambda^{-m(e)}
    = \lambda^{m(e)} K \lambda^{-m(e)}.
  \end{multline*}
  So if \(x\in\Rough(V)\) is \((K \cap M) U\)\nb-invariant, then~\(\idem{K_e}x\) is invariant under the subgroup \(\lambda^{m(e)} K \lambda^{-m(e)}\).  Conversely, if~\(\idem{K_e}x\) is invariant under \(\lambda^{m(e)} K \lambda^{-m(e)}\) for all sufficiently large \(e\in\N\), then~\(x\) is \((K \cap M) U\)\nb-invariant.  Thus
  \[
  \Jact_G^P(V)^{K \cap M}
  \cong \varprojlim_e V^{\lambda^{m(e)} K \lambda^{-m(e)}}
  \cong \varprojlim_m V^{\lambda^m K \lambda^m}.
  \]
  As for~\(\Jacr_G^{\opp{P}}\), the maps in this projective system are given by integration over the relevant subgroups, and replacing~\(m(e)\) by~\(m\) does not affect the projective limit because the sequence~\(\bigl(m(e)\bigr)_{e\in\N}\) is cofinal in~\(\N\).  As above, we apply the invertible elements~\(\lambda^{\pm m}\) to arrive at the desired isomorphism \(\Jact_G^P (V)^{K \cap M} \cong \varprojlim {}\idem{K \lambda K}\).

  The natural map \(\varprojlim {}\idem{K\lambda K}\to \varinjlim {}\idem{K\lambda K}\) maps an element~\((x_n)_{n\in\N}\) of the projective limit to the image of \(x_0 \in V^K\) in the inductive limit.  This is the same as the image of~\(x_k\) in the inductive limit for any \(k\in\N\).  Our normalisations \(m(0)=0=n(0)\) ensure that the resulting map \(\Jact_G^P(V)^{K\cap M}\to \Jacr_G^{\opp{P}}(V)^{K\cap M}\) is simply~\(\idem{K}\), as it should be.
\end{proof}

\begin{proposition}
  \label{pro:Jacquet_Stability_as_limit}
  If \(\Op^n\colon W\to W\) is stable for sufficiently large~\(n\), then the natural map \(\varprojlim \Op \to \varinjlim \Op\) is invertible.
\end{proposition}

\begin{proof}
  Since \(n\cdot\N\) is cofinal in~\(\N\), we have \(\varprojlim \Op^n = \varprojlim \Op\) and \(\varinjlim \Op^n = \varinjlim \Op\).  Hence we may assume without loss of generality that~\(\Op\) itself is stable.  By Lemma~\ref{lem:stable_decomposition}, the restriction of~\(\Op\) to \(\im \Op\) is bijective.  The inductive and projective limits of the constant systems
  \[
  \dotsb\to W \xrightarrow{\Op}
  W \xrightarrow{\Op}
  W \xrightarrow{\Op}
  W \xrightarrow{\Op}
  W \to\dotsb
  \]
  and
  \[
  \dotsb\to \im \Op \xrightarrow{\Op}
  \im \Op \xrightarrow{\Op}
  \im \Op \xrightarrow{\Op}
  \im \Op \xrightarrow{\Op}
  \im \Op \to\dotsb
  \]
  are isomorphic because the inclusion \(\im \Op\to W\) and the map \(\Op\colon W \to \im \Op\) shifting the diagrams by~\(1\) are inverse to each other as maps of projective or inductive systems.  Since \(\Op|_{\im \Op}\) is invertible by Lemma~\ref{lem:stable_decomposition}, \(\varprojlim \Op \cong \im \Op\) and \(\varinjlim \Op \cong \im \Op\), and the canonical map between them is the identity map on \(\im \Op.\)
\end{proof}

\begin{proposition}
  \label{pro:stability_to_adjointness}
  If \(\idem{K\lambda K}^n\) is stable for sufficiently large~\(n\), then the natural map \(\Jact_G^P(V)\to \Jacr_G^{\opp{P}}(V)\) restricts to an isomorphism \(\Jact_G^P(V)^{K\cap M}\cong \Jacr_G^{\opp{P}}(V)^{K\cap M}\).
\end{proposition}

\begin{proof}
  Combine Propositions \ref{pro:Jacquet_as_limit} and~\ref{pro:Jacquet_Stability_as_limit}.
\end{proof}

Thus the Second Adjointness Theorem~\ref{the:second_adjointness} follows if the Stabilisation Theorem holds for~\(\idem{K_e\lambda K_e}\) for some sequence of subgroups~\((K_e)\) with \(\bigcap K_e=\{1\}\).

For a general map \(\Op\colon W\to W\), stability of~\(\Op^n\) for sufficiently large~\(n\) is stronger than invertibility of the natural map \(\varprojlim \Op \to \varinjlim \Op\).  Nevertheless, we can deduce the Stabilisation Theorem from the Second Adjointness Theorem.

\begin{proposition}
  \label{pro:Jacquet_Stability_as_limit_converse}
  Let \(\Op_\lefts\) and~\(\Op_\rights\) denote the operators of left and right multiplication by \(\idem{K\lambda K}\) on \(\Hecke(G)\).  If the canonical maps \(\varprojlim \Op_\lefts \to \varinjlim \Op_\lefts\) and \(\varprojlim \Op_\rights \to \varinjlim \Op_\rights\) are invertible, then
  \begin{equation}
    \label{eq:ideal_stability}
    \idem{K\lambda K}^n\Hecke(G) = \idem{K\lambda K}^{n+1}\Hecke(G)
    \quad\text{and}\quad
    \Hecke(G)\idem{K\lambda K}^n = \Hecke(G)\idem{K\lambda K}^{n+1}
  \end{equation}
  for sufficiently large \(n\in\N\).  And~\eqref{eq:ideal_stability} implies that~\(\idem{K\lambda K}^n\) is stable on any smooth representation of~\(G\) on an \(\ring\)\nb-module.
\end{proposition}

\begin{proof}
  The class of \(\idem{K}\in \Hecke(G)^K\) in \(\varinjlim \Op_\lefts\) is the image of some element~\((x_{-n})_{n\in\N}\) of \(\varprojlim \Op_\lefts\).  Thus the \(\Op^n\)\nb-images of \(x_0\) and~\(\idem{K}\) agree for sufficiently large~\(n\), that is
  \[
  \idem{K\lambda K}^n = \Op_\lefts^n\idem{K} = \Op_\lefts^n x_0 = \Op_\lefts^{n+1}x_{-1}
  = \idem{K\lambda K}^{n+1} x_{-1}.
  \]
  The existence of such~\(x_{-1}\) is equivalent to \(\idem{K\lambda K}^n\Hecke(G) =\idem{K\lambda K}^{n+1}\Hecke(G)\).  A similar argument for~\(\Op_\rights\) yields \(\idem{K\lambda K}^n = y_{-1} \idem{K\lambda K}^{n+1}\) for some \(y_{-1}\in \Hecke(G)^K\) and \(\Hecke(G)\idem{K\lambda K}^n = \Hecke(G)\idem{K\lambda K}^{n+1}\) if the map \(\varprojlim \Op_\rights \to \varinjlim \Op_\rights\) is invertible.

  We write \(\pi\idem{K\lambda K}\) for the action of~\(\idem{K\lambda K}\) on a smooth \(G\)\nb-representation.  The inclusion \(\im \pi \idem{K \lambda K}^{n+1} \subseteq \im \pi\idem{K \lambda K}^n\) is trivial.  If \(\idem{K\lambda K}^n = \idem{K\lambda K}^{n+1} x_{-1}\) for some \(x_{-1}\in \ring\), then
  \[
  \im \pi\idem{K \lambda K}^{n+1}
  \supseteq \im\bigl( \pi\idem{K \lambda K}^{n+1}\cdot x_{-1} \bigr)
  = \im \pi\idem{K \lambda K}^n.
  \]
  Similarly, if \(\idem{K\lambda K}^n = y_{-1} \idem{K\lambda K}^{n+1}\) for some \(y_{-1}\in \Hecke(G)^K\), then
  \[
  \ker \pi\idem{K \lambda K}^{n+1}
  \subseteq \ker \bigl(y_{-1}\cdot \idem{K \lambda K}^{n+1}\bigr)
  = \ker \pi\idem{K \lambda K}^n,
  \]
  and \(\ker \pi\idem{K \lambda K}^{n+1} \supseteq \ker \pi\idem{K \lambda K}^n\) is trivial.  Thus~\eqref{eq:ideal_stability} implies that \(\pi\idem{K \lambda K}^n\) is stable.
\end{proof}

\begin{proposition}
  \label{pro:adjointness_to_stability}
  Let~\(K\) be a compact open subgroup in good position with respect to \(\{P,\opp{P}\}\) and let~\(\lambda\) be strictly positive with respect to \((\opp{P},M)\).  Let~\(G\) act on \(\Hecke(G)\) by the regular representation.  Assume that the natural maps
  \begin{align*}
    \Jact_G^P(\Hecke(G))^{K\cap M}&\to
    \Jacr_G^{\opp{P}}(\Hecke(G))^{K\cap M}\\
    \Jact_G^{\opp{P}}(\Hecke(G))^{K\cap M}&\to
    \Jacr_G^P(\Hecke(G))^{K\cap M}
  \end{align*}
  are invertible.  Then \(\idem{K\lambda K}^n\) is stable for sufficiently large~\(n\), on any smooth representation of~\(G\) on an \(\ring\)\nb-module.
\end{proposition}

\begin{proof}
  Using the notation \(\Op_\lefts\) and~\(\Op_\rights\) above, Proposition~\ref{pro:Jacquet_as_limit} identifies
  \begin{alignat*}{2}
    \varprojlim \Op_\lefts &\cong \Jact_G^P(\Hecke(G))^{K\cap M},&\qquad
    \varinjlim \Op_\lefts &\cong \Jacr_G^{\opp{P}}(\Hecke(G))^{K\cap M},\\
    \varprojlim \Op_\rights &\cong \Jact_G^{\opp{P}}(\Hecke(G))^{K\cap M},&\qquad
    \varinjlim \Op_\rights &\cong \Jacr_G^P(\Hecke(G))^{K\cap M}.
  \end{alignat*}
  Notice that \(\Op_\rights= \pi\idem{K\lambda^{-1}K}\) if~\(\pi\) is the right regular representation of~\(G\); if~\(\lambda\) is strictly positive with respect to \((\opp{P},M)\), then~\(\lambda^{-1}\) is strictly positive with respect to \((P,M)\).  Hence our assumptions imply the hypotheses of Proposition~\ref{pro:Jacquet_Stability_as_limit_converse}, which then yields the desired conclusion.
\end{proof}

Given a smooth \(G\)\nb-representation \((\pi,V)\), let \(p_U\colon V \to V / V(U) = \Jacr_G^P (V)\) be the quotient map and \(p_U^K\colon V^K \to \Jacr_G^P (V)^{K \cap M}\) its restriction to~\(V^K\).

\begin{theorem}
  \label{thm:7}
  Let~\(\mu\) be an element of the centre of~\(M\) that is strictly positive with respect to \((P,M)\).  Assume that \(\pi\idem{K\mu^nK}\) is stable.  For sufficiently large compact open subgroups \(C \subset U\) and \(\bar C \subset \bar U\):
  \begin{gather}
    \label{eq:8}
    \ker p_U^K = V^K \cap V(U) = V^K \cap \ker \pi (\idem{C}) =
    V^K \cap \ker \pi ( \idem{K \mu^n K}),\\
    \label{eq:9}
    V_*^K \defeq \pi (\idem{K \mu^n K}) V = \pi (\idem{K} \idem{\bar C}) V.
  \end{gather}
  Moreover, \(V^K = \ker p_U^K \oplus V_*^K\) and \(p_U\colon V_*^K \to \Jacr_G^P (V)^{K \cap M}\) is a bijection.
\end{theorem}

Conversely, \eqref{eq:8} and~\eqref{eq:9} imply that \(\ker \pi\idem{K\mu^NK}\) and \(\im \pi\idem{K\mu^NK}\) are independent of~\(N\) for \(N\ge n\), so that Theorem~\ref{thm:7} provides an equivalent reformulation of the Stabilisation Theorem.  This variant is closer to the Jacquet Lemma and generalises \cite{Casselman:Characters_Jacquet}*{Proposition 3.3}.

\begin{proof}
  By assumption,
  \[
  \ker \pi\idem{K \mu^N K}
  = \ker \pi\idem{K\mu^n K},\qquad
  \im \pi\idem{K\mu^N K}
  = \im \pi\idem{K \mu^nK}
  \]
  for all \(N\ge n\).  For any open subgroup \(\opp{C}\subseteq\opp{U}\), there is \(N\ge n\) such that \(\opp{C} \subseteq \mu^N (K\cap \opp{U}) \mu^{-N}\).  Hence the image of \(\pi(\idem{K}\idem{\opp{C}})\) contains
  \[
  \im \pi(\idem{K} \idem{\mu^N K \mu^{-N}})
  = \im \pi \idem{K \mu^N K}
  = \im \pi\idem{K\mu^n K}.
  \]
  If \(\opp{C} \supseteq \mu^n (K\cap \opp{U}) \mu^{-n}\), then \(\im \pi(\idem{K}\idem{\opp{C}}) \subseteq \im \pi(\idem{K \mu^n K})\) as well.  This yields~\eqref{eq:9}.

  Similarly, if~\(C\subseteq U\) is an open subgroup containing \(\mu^{-n} (K\cap U)\mu^n\), then \(\ker \pi(\idem{C} \idem{K}) = \ker \pi\idem{K\mu^nK}\).  This yields
  \[
  V^K \cap \ker \pi (\idem{C})
  = V^K \cap \ker \pi(\idem{K\mu^n K}).
  \]
  Proposition~\ref{pro:Jacquet_as_limit} for the opposite parabolic~\(\opp{P}\) implies
  \[
  \ker p_U^K
  = V^K\cap \bigcup \ker \pi\idem{K\mu^N K}
  = V^K \cap \pi\idem{K\mu^nK},
  \]
  hence~\eqref{eq:8}.  If \(\Op\colon W\to W\) is stable, then \(W\cong \ker \Op\oplus \im \Op\) by Lemma~\ref{lem:stable_decomposition}.  In particular, the stability of \(\pi\idem{K \mu^n K}\colon V^K\to V^K\) implies
  \[
  V^K = \ker \pi\idem{K \mu^n K} \oplus
  \im \pi\idem{K \mu^n K}.
  \]
  Thus \(V^K = \ker p_U^K \oplus V_*^K\) and \(p_U\colon V_*^K \to \Jacr_G^P(V)^{K \cap M}\) is injective.  Proposition~\ref{pro:Jacquet_as_limit} for~\(\opp{P}\) shows that the projection map \(V^K \to \Jacr_G^P(V)^{K \cap M}\) is surjective.  Hence so is its restriction to~\(V_*^K\).
\end{proof}

\section{Proof of the Stabilisation Theorem}
\label{sec:stabilisation}

In the first part of the proof, we will use the geometry of an apartment in the semisimple Bruhat--Tits building \(\bt\) of~\(G\) in order to reduce the assertion to the special case of semisimple groups of rank one.  The second part deals with this special case, using some basic results of Bruhat--Tits theory and combinatorics in certain finite subquotients of~\(G\).

We use the subgroups~\(\UC{x}^{(e)}\) for \(x \in \bt\), \(e \in \R_{\geq 0}\) constructed in \cite{Schneider-Stuhler:Rep_sheaves}*{Chapter~I}.  Their properties are listed in \cite{Meyer-Solleveld:Characters_growth}*{Section~5}.  In particular, they are in good position with respect to \(\{P,\opp{P}\}\) and satisfy \(g \UC{x}^{(e)} g^{-1} = \UC{g x}^{(e)}\) for all \(g \in G\).

Let~\(\ST\) be a maximal split torus of~\(M\) (hence of~\(G\)) and let~\(A_\ST\) be the apartment of~\(\bt\) corresponding to~\(\ST\).  Let \(\lambda \in \ST\) be an element that is central in~\(M\) and strictly positive with respect to \((\opp{P},M)\).  In particular, \(\lambda\)~acts as a translation on~\(A_\ST\).  Pick \(x_0 \in A_\ST\) and write \(x_n = \lambda^n x_0\) for \(n \in \Z\).  We assume that~\(x_0\) is generic, in the sense that the geodesic line through the points~\(x_n\) contains no point that lies in two different walls of~\(A_\ST\).  By Lemma~\ref{lem:1}
\begin{equation}
  \label{eq:Ux0}
  \idem{\UC{x_0}^{(e)} \lambda \UC{x_0}^{(e)}}^n \Hecke(G)
  = \idem{\UC{x_0}^{(e)}} \lambda^n \idem{\UC{x_0}^{(e)}} \Hecke(G)
  = \idem{\UC{x_0}^{(e)}} \idem{\UC{x_n}^{(e)}} \Hecke(G)
\end{equation}
and \(\Hecke(G)\idem{\UC{x_0}^{(e)} \lambda \UC{x_0}^{(e)}}^n = \Hecke(G)\idem{\UC{x_{-n}}^{(e)} \UC{x_0}^{(e)}}\).  We will show:

\begin{lem}
  \label{lem:stabilisation_special}
  The right ideals \(\idem{\UC{x_0}^{(e)}}\idem{\UC{x_n}^{(e)}} \Hecke(G)\) and left ideals \(\Hecke(G)\idem{\UC{x_{-n}}^{(e)}}\idem{\UC{x_0}^{(e)}}\) stabilise for \(n\in\N\) larger than some
explicitly computable bound.
\end{lem}

Lemma~\ref{lem:stabilisation_special} and the second half of Proposition~\ref{pro:Jacquet_Stability_as_limit_converse} show that \(\idem{\UC{x_0}^{(e)} \lambda \UC{x_0}^{(e)}}^n\) is stable on any smooth representation on an \(\ring\)\nb-module for any \(n\ge n_0\) for an explicitly computable~\(n_0\), not depending on the representation.  The subgroups \(\UC{x_0}^{(e)}\) for \(e\in\N\) form a neighbourhood basis of~\(1\).  Hence Lemma~\ref{lem:stabilisation_special} together with Propositions \ref{pro:stability_to_adjointness} and~\ref{pro:Jacquet_Stability_as_limit_converse} establishes the Second Adjointness Theorem.  The Second Adjointness Theorem together with Proposition~\ref{pro:adjointness_to_stability} yields the Stabilisation Theorem~\ref{the:stabilisation_general} in complete generality, for any compact open subgroup~\(K\) in good position with respect to \(\{P,\opp{P}\}\).  Hence the hypothesis of Theorem~\ref{thm:7} is always satisfied for some \(n\in\N\).  As a result, our main theorems all follow from Lemma~\ref{lem:stabilisation_special}.  But before we can prove it we need some supplementary technical results.

The building \(\bt\) is constructed via a valuated root datum on~\(G\), and it is this structure that we will use mostly.  The definition and construction of valuated root data is due to Fran\c{c}ois Bruhat and Jacques Tits \cites{Bruhat-Tits:Reductifs_I, Bruhat-Tits:Reductifs_II}.  We summarised some of their theory in \cite{Meyer-Solleveld:Characters_growth}*{Section~3}.

Let \(\valu\colon \F^\times \to \R\) be the discrete valuation of the field~\(\F\), let~\(\Phi\) be the root system of~\((G,\ST)\) and let~\(\Phr\) be the subset of reduced roots.  Since the Lie algebra of~\(G\) is, as an \(\ST\)\nb-representation, the direct sum of the Lie algebras of \(\opp{U}\),~\(M\) and~\(U\), there is a corresponding partition
\[
\Phi \cup \{0\} = \Phi_- \cup \Phi_0 \cup \Phi_+.
\]
The conditions on~\(\lambda\) translate to
\begin{alignat*}{2}
  \valu(\alpha(\lambda)) &<0&\qquad&\text{for \(\alpha \in \Phi_+\),}\\
  \valu(\alpha(\lambda)) &>0&\qquad&\text{for \(\alpha \in \Phi_-\), and}\\
  \valu(\alpha(\lambda)) &=0&\qquad&\text{for \(\alpha \in \Phi_0\).}
\end{alignat*}
The root subgroups \(\Un_\alpha \subset G\) for \(\alpha \in \Phi\) are filtered by compact subgroups~\(\Un_{\alpha,r}\) for \(r\in\R\).  These groups decrease when~\(r\) increases, and the set of jumps is \(\epal \Z\) for some \(\epal \in \R_{>0}\).  We put \(U_{\alpha} \defeq \{1\}\) if \(\alpha \notin \Phi \cup \{0\}\) and \(\Un_{2\alpha,r} \defeq \Un_{\alpha,r/2} \cap \Un_{2 \alpha}\) if \(\alpha, 2 \alpha \in \Phi\).  The centraliser~\(Z_G(S)\) of~\(S\) in~\(G\) plays the role of~\(\Un_0\), but we prefer not to use the latter notation.  The maximal compact subgroup~\(H\) of~\(Z_G(S)\) is filtered by normal compact open subgroups~\(H_r\) for \(r\in\R_{\ge 0}\), which are pro\nb-\(p\) when \(r > 0\).  We write \(\Un_{\alpha,r+} = \bigcup_{s > r} \Un_{\alpha,s}\) and \(H_{r+} = \bigcup_{s > r} H_s\).  The set of jumps of the filtration~\((H_r)\) is discrete, so \(H_{0+}\) is pro\nb-\(p\).

Let~\(P_X\) denote the pointwise stabiliser of a subset \(X \subseteq \bt\).  By construction,
\begin{equation}
  \label{eq:Uealpha}
  P_x \cap \Un_\alpha = \Un_{\alpha,-\alpha(x)} \quad \text{and} \quad
  \UC{x}^{(e)} \cap \Un_\alpha = \Un_{\alpha,(e - \alpha(x))+} \Un_{2 \alpha, (e - 2 \alpha (x))+},
\end{equation}
for \(x \in A_\ST\), where~\(\alpha \in \Phi\) is simultaneously regarded as a root of~\((G,\ST)\) and as an affine function on~\(A_\ST\).  Moreover, the multiplication maps
\begin{equation}
  \label{eq:uniqueFactorisation}
  \prod_{\alpha \in \Phr_+} \Un_{\alpha,-\alpha(x)} \to P_x \cap U \quad \text{and} \quad
  H_{e{+}} \! \times \! \prod_{\alpha \in \Phr} (\UC{x}^{(e)} \cap \Un_{\alpha}) \to \UC{x}^{(e)}
\end{equation}
are bijective for every ordering of~\(\Phr\) by \cite{Bruhat-Tits:Reductifs_I}*{Proposition 6.4.9} and \cite{Schneider-Stuhler:Rep_sheaves}*{Proposition I.2.7}.

\begin{example}
  \label{ex:SL2U}
  For \(G = \Sl_2 (\F)\), \(r \in \R\) and \(s \in \R_{>0}\), we have
  \begin{align*}
    \Un_{\alpha,r} &= \{ \stwomatrix{1}{x}{0}{1} : \valu (x) \ge r \}, \\
    \Un_{-\alpha,r} &= \{ \stwomatrix{1}{0}{x}{1} : \valu (x) \ge r \}, \\
    H_s &= \Bigl\{ \stwomatrix{1+x}{0}{0}{(1+x)^{-1}} : \valu (x) \ge s \Bigr\}.
  \end{align*}
  Let \(\integ = \{ x \in \F : \valu (x) \ge 0 \}\) and let \(\unif \in \integ\) be a uniformiser.  If~\(x\) is the origin of \(A_\ST \cong \R\) and \(e \in \N\), then \(P_x = \Sl_2 (\integ)\) and
  \[
  \UC{x}^{(e)} = \ker \bigl( \Sl_2 (\integ) \to \Sl_2 (\integ / \unif^{e+1} \integ) \bigr) =
  \biggl\{ \twomatrix{1+x}{y}{z}{\frac{1+yz}{1+x}} : x,y,z \in \unif^{e+1} \integ\biggr\}.
  \]
\end{example}

\subsection{Reduction to rank one}

Equations \eqref{eq:lambdaK+} and~\eqref{eq:lambdaK-} yield \(N\in\N\) with
\[
\UC{x_N}^{(e)} = \lambda^N \UC{x_0}^{(e)} \lambda^{-N} \supseteq P_{x_0} \cap U
\quad \text{and} \quad
P_{x_N} \cap \opp{U} = \lambda^N (P_{x_0} \cap \opp{U}) \lambda^{-N} \subseteq \UC{x_0}^{(e)}.
\]
Put \(y_0 = x_N \in A_\ST\) for the smallest such \(N \in \N\).  Let \(m \in \N\) be so large that \(\UC{x_m}^{(e)} \supseteq P_{y_0} \cap U\) and let \(y_1, \dotsc, y_{\ell (m)}\) be the points of the geodesic line segment \([y_0,x_m] \subset A_\ST\) where the isotropy groups jump, ordered from~\(x_N\) to~\(x_m\).  These are exactly the intersection points of the line segment \([y_0,x_m]\) with walls, as in Figure~\ref{fig:ys}.
\begin{figure}[h]
\includegraphics[width=12cm,height=5cm]{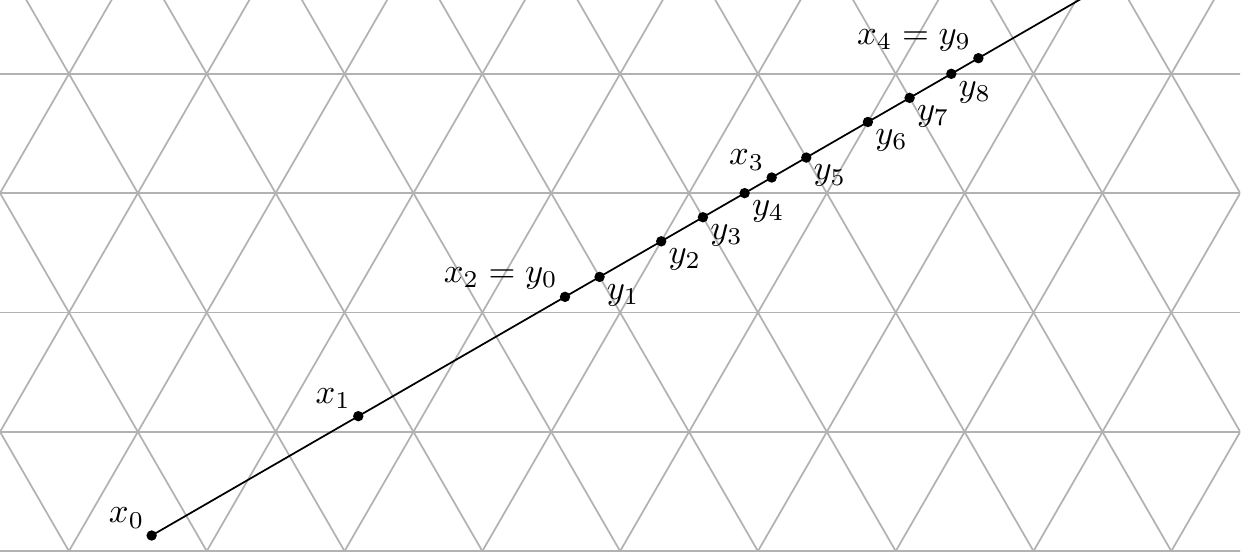}
  \caption{An example for an \(\widetilde{A_2}\)-building with \(N=2\), \(m=4\), \(\ell(m)=9\)}
  \label{fig:ys}
\end{figure}
Since \(x_0 \in A_\ST\) is generic, the wall containing~\(y_i\) is unique and contains a unique reduced positive root~\(\alpha_i\).  That is, \(-\alpha (y_i) \in \epsilon_{\alpha_i} \Z\) either for one root \(\alpha \in \Phr_+\) or for a pair of roots \(\alpha, 2\alpha \in \Phi_+\).

\begin{lem}
  \label{lem:filtration}
  We get a filtration
  \begin{equation}
    \label{eq:filtration1}
    P_{y_0} \cap U \subsetneq P_{y_1} \cap U \subsetneq \dotsb \subsetneq P_{y_{\ell(m)-1}} \cap U
    \subsetneq P_{y_{\ell(m)}} \cap U  = P_{x_m} \cap U.
  \end{equation}
  with the following properties:
  \begin{itemize}
  \item \(P_{y_i} \cap U \subseteq (P_{y_{i-1}} \cap U) \Un_{\alpha_i}\);
  \item \((P_{y_{i-1}} \cap U) \backslash (P_{y_i} \cap U) \cong \Un_{\alpha_i,\epsilon_{\alpha_i} - \alpha_i (y_i)} \backslash \Un_{\alpha_i,-\alpha_i (y_i)}\);
  \item \(P_{y_{i-1}} \cap U\) is normal in \(P_{y_i} \cap U\).
  \end{itemize}
\end{lem}

\begin{proof}
  Equation~\eqref{eq:uniqueFactorisation} implies that we get a filtration with the first two properties because \(\beta(y_{i-1}) =\beta(y_i)\) for \(\beta\in\Phr\setminus\{\alpha_i\}\) and \(\alpha_i(y_{i-1}) < \alpha_i(y_i)\).
  \cite{Bruhat-Tits:Reductifs_I}*{6.2.1} yields
  \[
  [\Un_{\alpha,r},\Un_{\alpha,r+ \epal}] \subseteq \Un_{2 \alpha, 2r + \epal}
  \subseteq \Un_{\alpha,r + \epal / 2} = \Un_{\alpha,r+ \epal}.
  \]
  Therefore~\(\Un_{\alpha,r+ \epal}\) is normal in~\(\Un_{\alpha,r}\) for all \(\alpha \in \Phi\) and \(r\in\R\).  Moreover, \([\Un_\alpha,\Un_\beta]\) for \(\alpha,\beta \in \Phi^+\) is contained in the group generated by the~\(\Un_{n\alpha + m\beta}\) for \(m,n \in \Z_{>0}\).  The third property of the filtration follows.
\end{proof}

We will use the filtration \eqref{eq:filtration1} to reduce the proof of Lemma~\ref{lem:stabilisation_special} to groups of rank one.  For \(\alpha \in \Phr\) let~\(G_\alpha\) be the subgroup of~\(G\) generated by \(\Un_{-\alpha} \cup \Un_\alpha\), and let \(T_\alpha \defeq G_\alpha \cap Z_G (\ST)\).  By \cite{Bruhat-Tits:Reductifs_I}*{6.2.1 and 6.3.4}, the normaliser of~\(T_\alpha\) in~\(G_\alpha\) is \(T_\alpha \sqcup M_\alpha\), where~\(M_\alpha\) is a single coset of~\(T_\alpha\).  Conjugation by elements of~\(M_\alpha\) exchanges \(\Un_\alpha\) and~\(\Un_{-\alpha}\). The groups \(T_\alpha^{(e)} \defeq T_\alpha \cap H_{e{+}} \; (e \in \R_{\geq 0})\) are pro\nb-\(p\), because \(H_{0+}\) is.

\begin{lem}
  \label{lem:rank_reduction}
  Suppose that
  \begin{equation}\label{eq:HeckeLalpha}
    \idem{\Un_{-\alpha,-r}} \idem{T_\alpha^{(e)} \Un_{\alpha,r + \epal}} \in
    \idem{\Un_{-\alpha,- r}} \idem{T_\alpha^{(e)} \Un_{\alpha,r}}
    \Hecke(G_\alpha),
  \end{equation}
  for all \(\alpha \in \Phr\) and all \(r \in \epal \Z\).  Then, for \(y_0\) and \(m \in \N\) as above:
  \[
  \idem{\UC{x_0}^{(e)}} \idem{\UC{x_m}^{(e)}} \Hecke(G) = \idem{\UC{x_0}^{(e)}} \idem{U \cap P_{y_0}} \Hecke(G).
  \]
\end{lem}

\begin{proof}
  The assumption \(\UC{x_m}^{(e)} \supseteq U \cap P_{y_0}\) implies \(\idem{\UC{x_0}^{(e)}} \idem{\UC{x_m}^{(e)}} \Hecke(G) \subseteq \idem{\UC{x_0}^{(e)}} \idem{U \cap P_{y_0}} \Hecke(G)\).  Let the \(y_i \in A_\ST\) be as in Lemma~\ref{lem:filtration} and consider the elements
  \[
  f_i \defeq \idem{\UC{x_0}^{(e)}} \idem{U \cap P_{y_i}} \in \Hecke(G).
  \]
  We have to show that \(f_0 \in \idem{\UC{x_0}^{(e)}} \idem{\UC{x_m}^{(e)}} \Hecke(G)\).  The groups~\(\UC{y_i}^{(e)}\) are well-placed with respect to \(\{P,\opp{P}\}\), and
  \begin{equation}
    \label{eq:barUM_compare}
    (\opp{U} M) \cap \UC{y_i}^{(e)} \subseteq (\opp{U} M) \cap \UC{x_0}^{(e)}.
  \end{equation}
  Thus
  \begin{alignat*}{2}
    f_i &= \idem{\UC{x_0}^{(e)}} \idem{U \cap P_{y_i}}
    &&= \bigidem{\UC{x_0}^{(e)}} \bigidem{(\opp{U} \cap \UC{y_i}^{(e)})
      (M \cap \UC{y_i}^{(e)}) (U \cap P_{y_i})}\\
    f_{\ell (m)} &= \idem{\UC{x_0}^{(e)}} \idem{\UC{x_m}^{(e)}} \idem{U \cap P_{x_m}}
    &&\in \idem{\UC{x_0}^{(e)}} \idem{\UC{x_m}^{(e)}} \Hecke(G).
  \end{alignat*}
  Hence the lemma follows if we show \(f_{i-1} \in f_i \Hecke(G)\) for all \(i \in \{1,2,\dotsc,\ell(m)\}\).

  We fix such an~\(i\) and abbreviate \(\alpha = \alpha_i \in \Phr_+\), \(r_i = -\alpha (y_i)\), and \(r_{i-1} = -\alpha (y_{i-1})\).  Notice that \(\Un_{\alpha,r_{i-1}} = \Un_{\alpha,r_i + \epal}\).  Using \eqref{eq:barUM_compare}, \eqref{eq:uniqueFactorisation}, and \(\Un_{-\alpha,-r_i} T_\alpha^{(e)} \subseteq (\opp{U} M) \cap \UC{y_i}^{(e)}\), we may rewrite
  \begin{multline}\label{eq:UPyi}
    \UC{x_0}^{(e)} (U \cap P_{y_i}) =
    \UC{x_0}^{(e)}(\opp{U}M \cap \UC{y_i}^{(e)}) (U \cap P_{y_i})
    \\= \UC{x_0}^{(e)}(\opp{U}M \cap \UC{y_i}^{(e)}) \Bigl( \Un_{-\alpha,-r_i} T_\alpha^{(e)}
    \prod_{\beta \in \Phr_+ \setminus \{\alpha\}} \Un_{\beta,-\beta (y_i)} \Bigr) \Un_{\alpha,r_i}.
  \end{multline}
{\large \textbf{Problem:} the next step is only correct when $e = 0$, that is, when $\UC{x_0}^{(e)}$ is a very large
pro-$p$-group. Therefore the proofs of this Lemma and of the Stabilisation Theorem are incomplete.}

  The expression between the large brackets satisfies the conditions of \cite{Bruhat-Tits:Reductifs_I}*{6.4.48}, so it is a group and changing the order of the factors does not make a difference.  In particular, we can use the ordering
  \begin{equation}
    \label{eq:separateAlpha1}
    \UC{x_0}^{(e)} (U \cap P_{y_i})
    = \UC{x_0}^{(e)}(\opp{U}M \cap \UC{y_i}^{(e)}) \biggl(\prod_{\beta \in \Phr_+ \setminus \{\alpha\}} \Un_{\beta,-\beta (y_i)} \biggr) \Un_{-\alpha,-r_i} T_\alpha^{(e)} \Un_{\alpha,r_i}.
  \end{equation}

  We may perform a similar computation for \(y_{i-1}\) instead of~\(y_i\).  The only differences in~\eqref{eq:uniqueFactorisation} involve the roots~\(\pm\alpha\), and \(\Un_{-\alpha,\alpha(y_{i-1})}\) is absorbed by \(\opp{U} \cap P_{y_0}\).  Thus
  \begin{equation}
    \label{eq:separateAlpha2}
    \UC{x_0}^{(e)} (U \cap P_{y_{i-1}})
    = \UC{x_0}^{(e)}(\opp{U}M \cap \UC{y_i}^{(e)}) \prod_{\beta \in \Phr_+ \setminus \{\alpha\}}
    \Un_{\beta,-\beta (y_i)} \Un_{-\alpha, - r_i} T_\alpha^{(e)} \Un_{\alpha,r_{i-1}}.
  \end{equation}
  Equations \eqref{eq:separateAlpha1} and~\eqref{eq:separateAlpha2} show that our assumption~\eqref{eq:HeckeLalpha} implies what we want.
\end{proof}

\begin{proof}[Proof of Lemma~\textup{\ref{lem:stabilisation_special}} assuming \eqref{eq:HeckeLalpha}]
The right ideals \(\idem{\UC{x_0}^{(e)}}\idem{\UC{x_m}^{(e)}} \Hecke(G)\) stabilise for sufficiently large \(m\in\N\) by Lemma~\ref{lem:rank_reduction}.  An analogous argument for~\(\opp{U}\) instead of~\(U\) shows that the right ideals \(\idem{\UC{x_0}^{(e)}}\idem{\UC{x_{-m}}^{(e)}} \Hecke(G)\) stabilise for sufficiently large \(m\in\N\).  Since the involution \(f^*(x) \defeq f(x^{-1})\) on \(\Hecke(G)\) maps the latter to the left ideal \(\Hecke(G)\idem{\UC{x_{-m}}^{(e)}}\idem{\UC{x_0}^{(e)}}\), we get Lemma~\ref{lem:stabilisation_special}.  Let us estimate how large~\(m\) must be chosen.  We need
  \[
  \UC{x_m}^{(e)} \cap U \supseteq P_{y_0} \cap U = P_{x_N} \cap U
  \quad\text{and}\quad
  \UC{y_0}^{(e)} \cap U = \UC{x_N}^{(e)} \cap U \supseteq P_{x_0} \cap U.
  \]
  By~\eqref{eq:uniqueFactorisation}, these are fulfilled if
  \[
  \epal + e - \alpha (\lambda^m x_0) \le -\alpha (\lambda^N x_0)
  \quad\text{and}\quad \epal + e - \alpha (\lambda^N x_0) \le -\alpha (x_0),
  \]
  for all \(\alpha \in \Phr_+\).  The action of~\(Z_G(\ST)\) on~\(A_\ST\) satisfies
  \[
  \alpha (\lambda^k x_0) - \alpha (x_0) = -\valu (\alpha (\lambda^k)) =
  -k \valu (\alpha(\lambda))
  \]
  for all \(k\in\Z\).  Recall also \(\valu(\alpha(\lambda))<0\) for \(\alpha\in\Phr_+\).  Hence we may rewrite the above inequalities as
  \[
  \epal + e \le (m-N) \abs{\valu \alpha (\lambda)}
  \quad\text{and}\quad
  \epal + e \le N \abs{\valu \alpha (\lambda)},
  \]
  for all \(\alpha \in \Phr_+\).  We choose the smallest~\(N\) satisfying the second condition, that is,
  \[
  N = \left\lceil \max_{\alpha\in\Phr_+}
    \frac{\epal + e}{\abs{\valu \alpha (\lambda)}} \right\rceil.
  \]
  We have stability whenever~\(m\) satisfies the first condition, that is, for
  \begin{equation}
    \label{eq:lower_bound_m}
    m \ge 2 \left\lceil \max_{\alpha\in\Phr_+}
      \frac{\epal + e}{\abs{\valu \alpha (\lambda)}} \right\rceil.\qedhere
  \end{equation}
\end{proof}

\begin{remark}
  \label{rem:lower_bound_optimal}
  The bound~\eqref{eq:lower_bound_m} is not far from the optimum.  Let
  \[
  m \defeq 2 \left\lfloor \min_{\alpha\in\Phi_+}
    \frac{\epal + e}{\abs{\valu \alpha (\lambda)}} \right\rfloor.
  \]
  A computation as above shows that \(\UC{x_0}^{(e)} \cup \UC{x_m}^{(e)}\) is contained in the group~\(P_{x_{m/2}}\), which is compact modulo the centre of~\(G\).  We may write down explicit functions on~\(P_{x_{m/2}}\) that belong to \(\idem{\UC{x_0}^{(e)}} \idem{\UC{x_{m-1}}^{(e)}} \Hecke(G)\) but not to \(\idem{\UC{x_0}^{(e)}} \idem{\UC{x_m}^{(e)}} \Hecke(G)\).  In particular, \eqref{eq:lower_bound_m} is optimal if~\(\Phi\) is of type~\(A_1\) and \((\epal + e)/ \abs{\valu \alpha (\lambda)}\) is an integer.
\end{remark}

\subsection{The rank one case}

It remains to prove~\eqref{eq:HeckeLalpha}, which is not only sufficient for Lemmas \ref{lem:rank_reduction} and~\ref{lem:stabilisation_special}, but also necessary for the stabilisation theorem to hold in the semisimple group~\(G_\alpha\) of rank one.  First we reformulate \eqref{eq:HeckeLalpha} in the setting of a finite group.

Pick \(y \in A_S\) with \(-\alpha (y) = r \in \epal \Z\).  Since~\(\UC{y}^{(e)}\) is an open normal subgroup of~\(P_y\) and \(P_y \cap G_\alpha\) is compact, the quotient
\[
L_{\alpha,e,r} \defeq  (P_y \cap G_\alpha) \bigm/ (\UC{y}^{(e)} \cap G_\alpha)
\]
is a finite group.  We let \(\Hecke (L_{\alpha,e,r})\) be its group algebra with coefficients in~\(\ring\).  Given any subset \(X \subseteq L_{\alpha,e,r}\), we put
\[
[X] = \sum_{x \in X} [x] \in \Hecke (L_{\alpha,e,r}).
\]
For subsets \(Y \subseteq P_y \cap G_\alpha\), we define
\begin{equation}
  \label{eq:[Y]}
  [Y] \defeq \bigl[Y\cdot (\UC{y}^{(e)} \cap G_\alpha) \bigm/ (\UC{y}^{(e)} \cap G_\alpha) \bigr].
\end{equation}
Now~\eqref{eq:HeckeLalpha} is implied by
\begin{equation}
  \label{eq:HeckeLer}
  [\Un_{-\alpha,-r}] [\Un_{\alpha,r+}] \in
  [\Un_{-\alpha,-r}] [\Un_{\alpha,r}] \Hecke (L_{\alpha,e,r}).
\end{equation}
Notice that~\(T_\alpha^{(e)}\) disappears, because it is contained in \(\UC{y}^{(e)} \cap G_\alpha\).  We may replace~\(\idem{\Un_{-\alpha,-r}}\) by~\([\Un_{-\alpha,-r}]\) because this is a pro-\(p\)-group and~\(p\)~is invertible in~\(\ring\).  The statement~\eqref{eq:HeckeLer} is stronger than \eqref{eq:HeckeLalpha} because we use a smaller algebra.  The proof of~\eqref{eq:HeckeLer} requires some rather technical preparations.

By \cite{Bruhat-Tits:Reductifs_I}*{6.4.48}, the product
\begin{equation}
  \label{eq:Galphars}
  \Un_{-\alpha,-r} T_\alpha^{(0)} \Un_{\alpha,r+}
\end{equation}
is a group and~\(T_\alpha^{(0)}\) normalises all three factors.  Moreover, each element of this group may be written uniquely as \(\bar{u} t u\) with \(\bar{u} \in \Un_{-\alpha,-r}\), \(t \in T_\alpha^{(0)}\), and \(u \in \Un_{\alpha,r+}\).  In particular, for every pair \((x,\bar{x}) \in \Un_{\alpha,r+} \times \Un_{-\alpha,-r}\) there exist unique \(\bar{u}(x,\bar{x}) \in \Un_{-\alpha,-r}\), \(u (x,\bar{x}) \in \Un_{\alpha,r+}\) and \(t (x,\bar{x}) \in T_\alpha^{(0)}\) such that
\begin{equation}
  \label{eq:xxyty}
  x \bar{x} = \bar{u}(x,\bar{x}) \, t(x,\bar{x}) \, u(x,\bar{x}).
\end{equation}
By the definition of root datum in \cite{Bruhat-Tits:Reductifs_I}*{6.1.1}, for every \(u \in \Un_\alpha \setminus \{1\}\) there are unique \(u', u'' \in \Un_{-\alpha}\) such that \(u' u u'' \eqdef m(u)\) lies in~\(M_\alpha\).  Similarly, for \(\bar u \in \Un_{-\alpha} \setminus \{1\}\) we get unique \(\bar{u}', \bar{u}'' \in \Un_{\alpha}\) with \(\bar{u}' \bar{u} \bar{u}'' = m(\bar{u}) \in M_\alpha\).  Moreover, \cite{Bruhat-Tits:Reductifs_I}*{6.2.1} yields
\begin{equation}
  \label{eq:u''}
  u', u'' \in \Un_{-\alpha,-r} - \Un_{-\alpha, \epal -r} \quad\text{if}\quad
  u \in \Un_{\alpha,r} - \Un_{\alpha, \epal + r}.
\end{equation}
The maps \(u \mapsto u'\) and \(u \mapsto u''\colon \Un_{\alpha} \setminus \{1\} \to \Un_{-\alpha} \setminus \{ 1 \}\) are bijective because \((u')'' = u = (u'')'\).
\begin{example}
  \label{ex:SL2T}
  Let \(G = G_\alpha = \Sl_2 (\F)\) and recall the notation from Example~\ref{ex:SL2U}.  Here
  \[
  T_\alpha^{(e)} = \bigl\{ \stwomatrix{1+x}{0}{0}{(1+x)^{-1}} : x \in \F, \valu (x) > e \bigr\}
  \quad \text{and} \quad
  M_\alpha = \bigl\{ \stwomatrix{0}{x}{-x^{-1}}{0} : x \in \F^\times \bigr\}.
  \]
  If \(u = \stwomatrix{1}{x}{0}{1} \in \Un_{\alpha} \setminus \{1\}\), then \(u' = u'' = \stwomatrix{1}{0}{-x^{-1}}{1}\) and \(m(u) = m(u') = \stwomatrix{0}{x}{-x^{-1}}{0}\).  In this case \eqref{eq:xxyty} is simply the equality
  \[
  \twomatrix{1}{x}{0}{1} \twomatrix{1}{0}{\bar x}{1} =
  \twomatrix{1}{0}{\frac{\bar x}{1+x \bar x}}{1} \twomatrix{1+x \bar x}{0}{0}{(1+ x \bar x)^{-1}}
  \twomatrix{1}{\frac{x}{1+x \bar x}}{0}{1}.
  \]
\end{example}

The group \eqref{eq:Galphars} projects onto the subgroup
\begin{equation}
  \label{eq:factorgroup}
  L^{(0)}_{\alpha,e,r} \defeq (\Un_{-\alpha,-r} T_\alpha^{(0)} \Un_{\alpha,r+}) \bigm/ 
(\UC{y}^{(e)} \cap G_\alpha)
\end{equation}
of~\(L_{\alpha,e,r}\), and the unique decomposition property of~\eqref{eq:Galphars} implies that~\(L^{(0)}_{\alpha,e,r}\) decomposes uniquely as
\[
\bigl( \Un_{-\alpha,-r} / ( \Un_{-\alpha} \cap \UC{y}^{(e)} ) \bigr)\times
\bigl( T_\alpha^{(0)} / T_\alpha^{(e)} \bigr)\times
\bigl( \Un_{\alpha,r+} / ( \Un_{\alpha} \cap \UC{y}^{(e)} ) \bigr).
\]
Moreover, \eqref{eq:xxyty}~remains valid, so that \(t(x,\bar{x}) \in T_\alpha^{(0)} / T_\alpha^{(e)}\) is well-defined for \(x \in \Un_{\alpha,r+} / \bigl( \Un_\alpha \cap \UC{y}^{(e)} \bigr)\) and \(\bar{x} \in \Un_{-\alpha,-r} / \bigl( \Un_{-\alpha} \cap \UC{y}^{(e)} \bigr)\).  A pivotal role will be played by the element
\begin{equation}
  S \defeq \sum_{x,\bar{x}} [t(x,\bar{x})] \in \Hecke \bigl( T_\alpha^{(0)} / T_\alpha^{(e)} \bigr),
\end{equation}
where the sum runs over \(x \in \Un_{\alpha,r+} / \bigl( \Un_\alpha \cap \UC{y}^{(e)} \bigr)\) and \(\bar{x} \in \Un_{-\alpha,-r} / \bigl( \Un_{-\alpha} \cap \UC{y}^{(e)} \bigr)\).

\begin{lem}
  \label{lem:S}
  The element~\(S\) has the following properties:
  \begin{enumerate}[label=\textup{(\alph{*})}]
  \item \(\bigl( [\Un_{-\alpha,-r}] [\Un_{\alpha,r+}] \bigr)^2 = [\Un_{-\alpha,-r}] S [\Un_{\alpha,r+}]\);
  \item \(S\) belongs to the centre of \(\Hecke \bigl(T_\alpha^{(0)} / T_\alpha^{(e)} \bigr)\);
  \item \([\Un_{-\alpha,-r}] S [\Un_{\alpha,r+}] \in [\Un_{-\alpha,-r}] [\Un_{\alpha,r}] \Hecke (L_{\alpha,e,r})\).
  \item If \(l\neq p\) is a prime, then the image of~\(S\) in \(\Hecke(T_\alpha^{(0)} / T_\alpha^{(e)},\Z/l)\) is invertible.
  \item \(S\) is invertible in \(\Hecke \big( T_\alpha^{(0)} / T_\alpha^{(e)}, \Z[\nicefrac1p] \big) \), and hence in \(\Hecke(T_\alpha^{(0)} / T_\alpha^{(e)},\ring)\) for any ring~\(\ring\) in which~\(p\) is invertible.
  \end{enumerate}
\end{lem}

\begin{proof}
  The decomposition in~\eqref{eq:xxyty} yields
  \begin{multline}
    \bigl( [\Un_{-\alpha,-r}] [\Un_{\alpha,r+}] \bigr)^2 = [\Un_{-\alpha,-r}]
    \sum_{x,\bar{x}} [\bar{u}(x,\bar{x}) \, t(x,\bar{x}) \, u(x,\bar{x}) ] [\Un_{\alpha,r+}] \\
    = [\Un_{-\alpha,-r}] \sum_{x,\bar{x}} [t(x,\bar{x}) ] [\Un_{\alpha,r+}] =
    [\Un_{-\alpha,-r}] S [\Un_{\alpha,r+}].
  \end{multline}
  This establishes~(a).

  For (b), we recall that~\(T_\alpha^{(0)}\) normalises~\(\Un_{\pm \alpha,s}\) for all \(s \in \R\).  So for all \(t \in T_\alpha^{(0)}\)
  \begin{multline}
    [\Un_{-\alpha,-r}] t S t^{-1} [\Un_{\alpha,r+}]
    = t [\Un_{-\alpha,-r}] S [\Un_{\alpha,r+}] t^{-1}
    = t \bigl( [\Un_{-\alpha,-r}] [\Un_{\alpha,r+}] \bigr)^2 t^{-1}
    \\= \bigl( [\Un_{-\alpha,-r}] [\Un_{\alpha,r+}] \bigr)^2
    = [\Un_{-\alpha,-r}] S [\Un_{\alpha,r+}].
  \end{multline}
  Together with the unique decomposition property for \(L^{(0)}_{\alpha,e,r}\), this implies \(t S t^{-1} = S\) for all \(t \in T_\alpha^{(0)}\), which is equivalent to~\(S\) being central in \(\Hecke \bigl( T_\alpha^{(0)} / T_\alpha^{(e)} \bigr)\).

  To prove~(c), we begin with the trivial equality
  \begin{multline}
    \label{eq:UUU}
    [\Un_{-\alpha,-r}] [\Un_{\alpha,r}] [\Un_{-\alpha,-r}]
    = [\Un_{-\alpha,-r}] [\Un_{\alpha, r+}] [\Un_{-\alpha,-r}]
    \\+ [\Un_{-\alpha,-r}] [\Un_{\alpha,r} \setminus \Un_{\alpha, r+}] [\Un_{-\alpha,-r}].
  \end{multline}
  We claim that the second summand equals
  \begin{equation}
    \label{eq:UUM}
    [\Un_{-\alpha,-r}] [\Un_{\alpha,r}] \sum_{u \in (\Un_{\alpha,r} - \Un_{\alpha, r+}) /
      (\Un_\alpha \cap \UC{y}^{(e)}) } [m(u)].
  \end{equation}
  Indeed, we may rewrite any \(\bar u_1 u \bar u_2 \in \Un_{-\alpha,-r} \bigl(\Un_{\alpha,r} \setminus \Un_{\alpha, r+} \bigr) \Un_{-\alpha,-r}\) as
  \begin{multline*}
    \bar u_1 u \bar u_2
    = \bar{u}_1 (u')^{-1} u' u u'' (u'')^{-1} \bar{u}_2
    = \bar{u}_1 (u')^{-1} m(u) (u'')^{-1} \bar{u}_2
    \\= \bar{u}_1 (u')^{-1} \bigl( m(u) (u'')^{-1} \bar{u}_2 m(u)^{-1} \bigr) m(u).
  \end{multline*}
  For fixed~\(u\) these elements run precisely once through \(\Un_{-\alpha,-r} \Un_{\alpha,r} m(u)\) when \(\bar u_1\) and~\(\bar u_2\) run through~\(\Un_{-\alpha,-r}\), proving~\eqref{eq:UUM}.  Subtracting~\eqref{eq:UUM} from~\eqref{eq:UUU} yields
  \begin{equation}
    \label{eq:3.10}
    [\Un_{-\alpha,-r}] [\Un_{\alpha, r+}] [\Un_{-\alpha,-r}]
    \in [\Un_{-\alpha,-r}] [\Un_{\alpha,r}] \Hecke (L_{\alpha,e,r}).
  \end{equation}
  Finally, we multiply this element from the right with \([\Un_{\alpha, r+}]\) and use~(a) to get~(c).

  To establish~(d), we must show that the operator of multiplication by the central element~\(S\) is invertible on \(\Hecke(T_\alpha^{(0)} / T_\alpha^{(e)},\Z/l)\).  Since the latter is a finite-dimensional vector space over~\(\Z/l\), it suffices to prove injectivity.  Assume \(Sf=0\).  Then also \(Sf^2=0\) and hence
  \begin{multline*}
    0 = [\Un_{-\alpha,-r}] Sf^2 [\Un_{\alpha,r+}]
    = [\Un_{-\alpha,-r}] S [\Un_{\alpha,r+}] f^2
    \\= \bigl([\Un_{-\alpha,-r}] [\Un_{\alpha,r+}]\bigr)^2 f^2
    = \bigl([\Un_{-\alpha,-r}] f [\Un_{\alpha,r+}]\bigr)^2.
  \end{multline*}
  This computation in \(\Hecke(L^{(0)}_{\alpha,e,r},\Z/l)\) uses that \([\Un_{\alpha,r+}]\) and \([\Un_{-\alpha,-r}]\) commute with all elements of \(\Hecke(T_\alpha^{(0)},\Z/l)\).  Since \([\Un_{-\alpha,-r}] f [\Un_{\alpha,r+}]\) is nilpotent, the operator of left convolution with \([\Un_{-\alpha,-r}] f [\Un_{\alpha,r+}]\) on \(\Hecke(L^{(0)}_{\alpha,e,r},\Z/l)\) has vanishing trace.  This trace is~\(\abs{L^{(0)}_{\alpha,e,r}}\) times the coefficient of \([\Un_{-\alpha,-r}] f [\Un_{\alpha,r+}]\) at~\(1\).  By the unique decomposition property of \(L^{(0)}_{\alpha,e,r}\), this coefficient at~\(1\) is equal to the coefficient of~\(f\) itself at~\(1\).  Since~\(L^{(0)}_{\alpha,e,r}\) is a \(p\)\nb-group, its order is invertible modulo~\(l\).  Thus we get \(f(1)=0\).

  Since the condition \(Sf=0\) defines a right ideal, the same reasoning may be applied to~\(f [t]\), where \(t\in T_\alpha^{(0)} / T_\alpha^{(e)}\).  We get \(0 = (f [t])(1) = f(t^{-1})\).  Thus \(f=0\), that is, multiplication by~\(S\) is invertible on \(\Hecke(L^{(0)}_{\alpha,e,r},\Z/l)\).  This finishes the proof of~(d).

  For~(e), consider the operator of multiplication by~\(S\) on \(\Hecke \big( T_\alpha^{(0)} / T_\alpha^{(e)},\Z[\nicefrac1p] \big)\).  Since the latter is a finite-dimensional free \(\Z[\nicefrac1p]\)-module, Cramer's rule yields the invertibility of~\(S\) if the determinant of this map is invertible in \(\Z[\nicefrac1p]\), that is, not divisible by any prime \(l\neq p\).  But if~\(l\) would divide the determinant of this map, then multiplication by~\(S\) on \(\Hecke (T_\alpha^{(0)} / T_\alpha^{(e)},\Z/l)\) would not be invertible, contradicting~(d).  Hence~\(S\) is invertible in \(\Hecke \big( T_\alpha^{(0)} / T_\alpha^{(e)},\Z[\nicefrac1p] \big)\), establishing~(e).
\end{proof}

Finally, we are able to prove our main results.  Since~\(T_\alpha^{(0)}\) normalises~\(\Un_{\alpha,s}\) for all~\(s\), we have \([\Un_{\alpha,r+}] S = S[\Un_{\alpha,r+}]\).  Since~\(S\) is invertible, we also get \([\Un_{\alpha,r+}] = S[\Un_{\alpha,r+}]S^{-1}\).  Lemma~\ref{lem:S}.c yields
\[
[\Un_{-\alpha,-r}] [\Un_{\alpha,r+}] = [\Un_{-\alpha,-r}] S [\Un_{\alpha,r+}] S^{-1}
\in [\Un_{-\alpha,-r}] [\Un_{\alpha,r}] \Hecke (L_{\alpha,e,r}).
\]
This establishes~\eqref{eq:HeckeLer}.  We already observed that~\eqref{eq:HeckeLer} implies~\eqref{eq:HeckeLalpha} and that this finishes the proofs of our main theorems, the Stabilisation Theorem~\ref{the:stabilisation_general} and the Second Adjointness Theorem~\ref{the:second_adjointness}.
\vspace{4mm}

\section{Consequences of Second Adjointness}
\label{sec:consequences}

The Second Adjointness Theorem has many noteworthy consequences.  Several of these are due to Jean-Fran\c{c}ois Dat~\cite{Dat:Finitude}.  With our proof of the Second Adjointness Theorem, they have become valid in greater generality.

As before, we work in the category of smooth representations on \(\ring\)\nb-modules, where~\(\ring\) is a unital ring in which~\(p\) is invertible.  We call a representation of~\(G\) on an \(\ring\)\nb-module \emph{projective}, \emph{finitely generated}, or \emph{finitely presented} if it is finitely generated or presented as a module over \(\Hecke(G)\).

\begin{lem}
  \label{lem:preserve-finite-type}
  The functors \(\Jaci_P^G\) and~\(\Jacr_G^P\)
  \begin{enumerate}[label=\textup{(\alph{*})}]
  \item are exact and commute with arbitrary colimits;
  \item preserve projective representations;
  \item preserve finitely generated representations;
  \item preserve finitely presented representations;
  \item preserve the property that a representation~\(V\) is generated by~\(V^K\) for some open subgroup~\(K\).
  \end{enumerate}
\end{lem}

\begin{proof}
  It is easy to see that both \(\Jaci_P^G\) and~\(\Jacr_G^P\) are exact and commute with direct sums.  This implies that they commute with arbitrary colimits.

  Let \(F\colon \Cat\to\Cat'\) be a functor between two Abelian categories with an exact right adjoint functor~\(G\).  If~\(P\) is projective in~\(\Cat\), then \(Y\mapsto \Cat'(F(P),Y) \cong \Cat(P,G(Y))\) is an exact functor on~\(\Cat'\), that is, \(F(P)\) is projective in~\(\Cat'\).  By the First and Second Adjointness Theorem, the functors \(\Jaci_P^G\) and~\(\Jacr_G^P\) have the right adjoint functors \(\Jacr_G^P\) and~\(\Jaci_{\tilde{P}}^G\), respectively.  Since these are both exact by~(a) and the two adjointness theorems, we get~(b) for both functors.

  Finitely generated modules may be described categorically: a module~\(X\) over \(\Hecke(G)\) or \(\Hecke(M,\ring)\) is finitely generated if and only if for every increasing net of submodules~\((Y_i)\) of a module~\(Y\) with \(Y=\bigcup Y_i\), we have
  \begin{equation}
    \label{eq:fg_categorical}
    \Hom\bigl(X,\bigcup Y_i\bigr) = \bigcup \Hom(X,Y_i).
  \end{equation}
  A functor between module categories with a right adjoint preserves property~\eqref{eq:fg_categorical} provided its adjoint maps injective maps to injective maps and preserves unions.  Since the right adjoints of \(\Jaci_P^G\) and~\(\Jacr_G^P\) have these properties by~(a), we get~(c).

  A module~\(X\) is finitely presented if and only if the functor \(Y\mapsto \Hom(X,Y)\) commutes with arbitrary inductive limits (also called filtered colimits).  Thus~(a) and the two adjointness theorems imply~(d).

  (e)~is proved in \cite{Dat:Finitude}*{Lemme 4.6} and (without using second adjointness) in
\cite{Meyer-Solleveld:Characters_growth}*{Proposition 5.8}.
\end{proof}

Let~\(\widetilde V\) be the contragredient representation of~\(V\), that is, the smooth part of the algebraic dual \(\Hom_\ring (V,\ring)\).  It is easily seen that
\begin{equation}
  \label{eq:hom-dual}
  \Hom_G(Y,\widetilde{V}) \cong \Hom_G(V,\widetilde{Y})
  \qquad \text{for all \(V,Y \in \Mod{\Hecke(G)}\),}
\end{equation}
while \cite{Vigneras:l-modulaires}*{II.2.1.vi} yields
\begin{equation}
  \label{eq:induction-dual}
  \widetilde{\Jaci_P^G (W)} \cong \Jaci_P^G (\widetilde{W})
  \qquad \text{for all \(W \in \Mod{\Hecke(M)}\).}
\end{equation}

\begin{theorem}
  \label{thm:Jacquet-restriction-dual}
  There is a natural isomorphism \(\widetilde{\Jacr_G^P (V)} \cong \Jacr_G^{\opp{P}} (\widetilde V)\).
\end{theorem}

\begin{proof}
  Equations \eqref{eq:hom-dual} and~\eqref{eq:induction-dual} and the First and Second Adjointness Theorems yield natural isomorphisms
  \begin{multline*}
    \Hom_M \bigl( W,\widetilde{\Jacr_G^P (V)} \bigr) \cong \Hom_M \bigl( \Jacr_G^P (V), \widetilde W \bigr)
    \cong \Hom_G \bigl( V,\Jaci_P^G (\widetilde W) \bigr)\\
    \cong \Hom_G \bigl( V,\widetilde{\Jaci_P^G (W)} \bigr) \cong \Hom_G \bigl( \Jaci_P^G (W),
    \widetilde{V} \bigr) \cong \Hom_M \bigl( W, \Jacr_G^{\opp{P}} (\widetilde V) \bigr).
  \end{multline*}
  Since this holds for all \(W \in \Mod{\Hecke(M)}\), the required natural isomorphism exists by the Yoneda Lemma.
\end{proof}

The isomorphism in Theorem~\ref{thm:Jacquet-restriction-dual} is described explicitly in \cite{Bernstein:Second_adjointness} and \cite{Bushnell:Localization_Hecke}*{Section 5}.  The Second Adjointness Theorem follows from Theorem~\ref{thm:Jacquet-restriction-dual} --~that was Bernstein's strategy in~\cite{Bernstein:Second_adjointness}.  Namely, a rearrangement of the above proof shows that
\[
\Hom_M \bigl(W, \Jacr_G^{\opp{P}} (\widetilde V)\bigr) \cong \Hom_G \bigl(\Jaci_P^G (W),\widetilde{V}\bigr)
\]
and then Bernstein uses a trick to replace~\(\widetilde{V}\) by an arbitrary smooth representation.

A \(\Hecke(G)\)-module~\(V\) is called \emph{locally Noetherian} if all \(\Hecke(G/\!\!/K)\)-submodules of~\(V^K\) are finitely generated, for each compact open subgroup~\(K\).

\begin{theorem}
  \label{thm:Hecke-algebra-Noetherian}
  Suppose that~\(\ring\) is a unital Noetherian \(\Z[\nicefrac1p]\)-algebra.
  \begin{enumerate}[label=\textup{(\alph{*})}]
  \item Every finitely generated smooth \(G\)\nb-representation is locally Noetherian.
  \item The algebra \(\Hecke(G/\!\!/K)\) is Noetherian, for every open subgroup \(K \subseteq G\).
  \end{enumerate}
\end{theorem}

\begin{proof}
  See Proposition~4.3 and Corollaire~4.4 in~\cite{Dat:Finitude}.
\end{proof}

It is quite hard to prove Theorem~\ref{thm:Hecke-algebra-Noetherian}, even with the Second Adjointness Theorem available.

\begin{theorem}[\cite{Dat:Finitude}*{Corollaire 4.5}]
  \label{thm:modG-Noetherian}
  Suppose that~\(\ring\) is a unital Noetherian \(\Z[\nicefrac1p]\)-algebra.  The category \(\Mod{\Hecke(G)}\) is Noetherian, in the sense that every submodule of a finitely generated module is again finitely generated.  Equivalently, all finitely generated smooth modules are finitely presented.
\end{theorem}

\begin{proof}
  Let~\(\Mod{e}\) be the category of all smooth \(\Hecke(G)\)-modules that are generated by the sum of their \(\UC{x}^{(e)}\)-invariants for all vertices in the building.  It is shown in \cite{Meyer-Solleveld:Resolutions}*{Section~3} that this is a Serre subcategory of \(\Mod{\Hecke(G)}\).  In particular, subrepresentations of representations in \(\Mod{e}\) again belong to \(\Mod{e}\).  Furthermore, \(\Mod{e}\) is equivalent to the category of \(u_\Delta^{(e)} \Hecke(G) u_\Delta^{(e)}\)-modules for a certain idempotent \(u_\Delta^{(e)}\in \Hecke(G)\) constructed from the idempotents \(\idem{\UC{x}^{(e)}}\) for the vertices of a chamber in the building.

  Now we may deduce the assertion from Theorem~\ref{thm:Hecke-algebra-Noetherian}.  We must show that any submodule~\(Y\) of a finitely generated smooth \(\Hecke(G)\)-module~\(V\) is again finitely generated.  Since~\(V\) is finitely generated, it is generated by the subspace of \(\UC{x}^{(e)}\)-invariants for sufficiently large~\(e\) and hence belongs to \(\Mod{e}\).  So does~\(Y\).  Thus~\(Y\) is generated by \(Y^K\) for some sufficiently small compact open subgroup.  But \(Y^K\subseteq V^K\) is finitely generated as a \(\Hecke(G/\!\!/K)\)-module by Theorem~\ref{thm:Hecke-algebra-Noetherian}.  Then~\(Y\) is finitely generated as a \(\Hecke(G)\)-module.
\end{proof}

\end{document}